\documentclass[a4paper]{article}
\usepackage[utf8]{inputenc}
\usepackage[T1]{fontenc}
\usepackage{indentfirst}
\usepackage[a4paper,top=2.54cm,bottom=2.0cm,left=3.0cm,right=3.54cm]{geometry}

\usepackage[round,comma]{natbib}

\usepackage{hyperref}
\usepackage{float}

\usepackage{amsmath}
\usepackage{amssymb}
\usepackage{amsthm}

\usepackage{tikz}
\usetikzlibrary{arrows.meta}
\tikzset{
    circ/.style={draw, circle,inner sep=0pt,minimum size=8mm, font=\scriptsize},
    triangle/.tip={Computer Modern Rightarrow[open,angle=120:3pt]}
}

\newtheorem{theorem}{Theorem}
\newtheorem{corollary}[theorem]{Corollary}
\newtheorem{lemma}[theorem]{Lemma}
\newtheorem{proposition}[theorem]{Proposition}
\newtheorem{remark}[theorem]{Remark}

\def\P{{\mathbb P}}
\newcommand{\E}{\mathbb{E}}

\newcommand{\N}{\mathcal{A}_N}
\newcommand{\W}{\mathcal{W}_N}

\newcommand{\st}{\mathcal{S}_N}
\newcommand{\auxst}{\tilde{\mathcal{S}}_N}

\newcommand{\ladder}{\mathcal{L}_N}

\newcommand{\nlist}{\vec{0}_N}

\title{\textbf{\Large{ Metastability in a Stochastic  System of \\ Spiking Neurons with Leakage}}}

\author{Kádmo de S. Laxa}

\date{November 16, 2022}

\begin{document}

\maketitle
\begin{abstract}
We consider a finite system of interacting point processes with memory of variable length modeling a finite but large network of spiking neurons with two different leakage mechanisms. Associated to each neuron there are two point processes, describing its successive spiking and leakage times. For each neuron, the rate of the spiking point process is an exponential function of its membrane potential, with the restriction that the rate takes the value $0$ when the membrane potential is $0$. At each spiking time, the membrane potential of the neuron resets to $0$, and simultaneously, the membrane potentials of the other neurons increase by one unit. 
The leakage can be modeled in two different ways. In the first way,
at each occurrence time of the leakage point process associated to a neuron, the membrane potential of that neuron is reset to $0$, with no effect on the other neurons.  In the second way, if the membrane potential of the neuron is strictly positive, at each occurrence time of the leakage point process associated to that neuron, its membrane potential decreases by one unit, with no effect on the other neurons. In both cases, the leakage point process of the neurons has constant rate.
For both models, we prove that the system has a metastable behavior as the population size diverges.
This means that the time at which the system gets trapped by the list of null membrane
potentials suitably re-scaled converges to a mean one exponential random time.
\end{abstract}

\vspace{0.2cm}
 
 
 \vspace{0.2cm}
 

\section{Introduction} \label{sec:intro}

We study a system of interacting point processes with memory of variable length modeling a finite but large network of spiking neurons with two different ways to model the leakage effect. We prove that when the population size diverges the system has a metastable behavior.

The system we consider can be informally described as follows.
Each neuron is associated to two point processes. The first point process indicates the successive spiking times of the neuron. The rate of this point process is an exponential function of the membrane potential of the neuron, with the restriction that the rate takes the value $0$ when the membrane potential is $0$. When a neuron spikes, its membrane potential resets to $0$, and simultaneously, the membrane potentials of the other neurons increase by one unit.

The second point process associated to each neuron indicates its successive leakage times. The leakage effect can be modeled in two different ways. In the first way, at each leakage time of the neuron, its membrane potential is reset to $0$, with no effect on the other neurons membrane potentials. In the second way, if the membrane potential of the neuron is strictly positive, at each leakage time of this neuron, its membrane potential decreases by one unit, with no effect on the other neurons membrane potentials. For both models, this point process has a fixed constant rate.

The first way to model the leakage effect was considered in  \cite{gl2} with the important difference that, besides considering only binary spiking rates, it also consider that the set of neurons is represented by the set of all integers, with each neuron interacting only with its two neighbors.

The second way to model the leakage effect was considered in \cite{amarcos} with the important difference that, besides considering only binary spiking rates, it also only consider interaction graphs that are regular trees.
We thank an anonymous reviewer for suggest to consider also the second way to model the leakage effect.

Let us now informally present our results. 
For any initial configuration of membrane potentials, 
the number of spiking and leakage times of the system is finite. Moreover, 
the process gets trapped after a finite time in the configuration in which the membrane potentials of all neurons are $0$. This is the content of Theorem \ref{teo: extinction} for the system with the first leakage mechanism and the content of Theorem \ref{teo: extinction2} for the system with the second leakage mechanism. 

Let us suppose that the system starts with a configuration in which a sufficiently large set of neurons have strictly positive membrane potential. With such a starting point, as the number of neurons of the system diverges, the system instantaneously reaches a set of configurations in which all neurons but one have strictly positive membrane potentials and these membrane potentials are all different. 
The system are in this set with probability approaching to $1$, for any instant before it gets trapped as the number of neurons of the system diverges. 
This is the content of Theorem \ref{massconcentration} for the system with the first leakage mechanism and the content of Theorem \ref{massconcentration2} for the system with the second leakage mechanism.

The system has a metastable behavior, namely the time at which it gets trapped in the null membrane potentials configuration re-normalized by its mean value converges in distribution to a mean $1$ exponential random time as the population size diverges. 
This is the content of Theorem \ref{metastable} for the system with the first leakage mechanism and the content of Theorem \ref{metastable2} for the system with the second leakage mechanism.
Theorems \ref{metastable} and \ref{metastable2} assume that the system starts with the same type of initial configuration considered in Theorems \ref{massconcentration} and \ref{massconcentration2}. This initial configuration condition prevents the system to be immediately  attracted by the null configuration.



To put our article in perspective, let us briefly recall some results recently published in other articles.
In article \cite{gl2} it was proven that there exists a critical value for the leakage rate such that the system has either one or two extremal invariant measures when the leakage rate is either greater or
smaller than the critical value, respectively. For the same model considered in article \cite{gl2},
it was proven by \cite{mandre1} that for a finite system with a sufficiently small leakage rate, the system displays a \text{metastable} behavior when the number of the neurons diverges (see also \cite{mandre2} and \cite{mandre3}). 


In article \cite{amarcos} it was proven that there exist two critical values for the leakage rate such that the system exhibit three different behaviors. In the first case, each fixed neuron has positive probability to spike infinitely many times. In the second case, the system has a positive probability of never goes extinct but each neuron eventually stops spiking. In the third case, the neural spiking activity goes extinct with probability one.

Both models considered here belong to the class of systems of interacting point process with memory of variable length that was introduced in discrete time 
by \cite{glmodel} 
and in continuous time by 
\cite{gl4} to model systems of spiking neurons.
The metastable behavior of systems of interacting point processes with memory of variable length was also analyzed by \cite{taille,evamonm} and \cite{galveslaxa}. Other aspects of systems of interacting point processes with memory of variable length in this class of models was considered in several articles, including \cite{ ostduarte2,lud,ostduarte1, evafour,glarticle, karina,baccelli1,ostduarte3,gl3,baccelli2,amarcos,baccelli3,lebow} and \cite{desantis}. For a self-contained and neurobiological motivated presentation of this class of variable length memory models for system of spiking neurons, both
in discrete and continuous time, we refer the reader to \cite{gb}.
 
The notion of metastability considered here is inspired by the so called \textit{pathwise approach to metastability}
introduced by \cite{cassandro}. 
For more references and an introduction to the topic, we refer the reader to \cite{metaest2,metaest3} and \cite{metaest1}.

This article is organized as follows. In Section \ref{cap:model} we present the definitions, basic and extra notation and state the main results.
In Section \ref{sec:teoextinction} we prove Theorem \ref{teo: extinction}.
In Section \ref{sec:auxiliary} 
we present a coupling construction and prove some auxiliary results. 
In Sections \ref{sec:teo1} and \ref{sec:teo2} we prove Theorems \ref{massconcentration} and \ref{metastable}, respectively. In Section \ref{sec:hatprocess} we prove Theorems \ref{teo: extinction2}, \ref{massconcentration2} and \ref{metastable2}.

\section{Definitions, notation and main results}
\label{cap:model}

Let $\N=\{1,2,...,N\}$ be the set of neurons, with $N \geq 2$ and denote
$$
\st=\big\{u=(u(a): a \in \N) \in \{0,1,2,\ldots\}^N:\min\{u(a):a \in \N\}=0 \big\}
$$
the set of lists of membrane potentials. 

We want to describe the time evolution of the list of membrane potentials of a system of spiking neurons.
To do this, for any neuron $a \in \N$, we define the maps $\pi^{a,*}$, $\pi^{a,\dagger}$ and $\hat{\pi}^{a,\dagger}$ on $\st$
as follows. For any $u\in \st$,  
$$
\pi^{a,*}(u)(b)=
\begin{cases}
u(b)+
1&\text{, if } b \neq a, \\ 
0 &\text{, if } b  = a,
\end{cases}
$$
$$
\pi^{a,\dagger}(u)(b)=
\begin{cases}
u(b) &\text{, if } b\neq a, \\
0 &\text{, if } b  = a,
\end{cases}
$$
$$
\hat{\pi}^{a,\dagger}(u)(b)=
\begin{cases}
u(b) &\text{, if } b\neq a, \\
u(b)-1 &\text{, if } b  = a \text{ and } u(b) \geq 1, \\
0 &\text{, if } b  = a \text{ and } u(b) =0.
\end{cases}
$$

The map $\pi^{a,*}$ represents the effect of a spike of neuron $a$ in the system. When we apply the map $\pi^{a,*}$, the membrane potential of neuron $a$ resets to $0$ and the membrane potentials of all the other neurons increase by one unit. 

The map $\pi^{a,\dagger}$ represents the leakage on the membrane potential of neuron $a$ following the first way to model the leakage effect. When we apply the map  $\pi^{a,\dagger}$, the membrane potential of neuron $a$ resets to $0$ and the membrane potentials of all the other neurons remain the same. 

The map $\hat{\pi}^{a,\dagger}$ represents the leakage on the membrane potential of neuron $a$ following the second way to model the leakage effect. When we apply the map  $\hat{\pi}^{a,\dagger}$, if neuron $a$ has a strictly positive membrane potential its membrane potential decreases by one unit and the membrane potentials of all the other neurons remain the same. 





The time evolution of the system of neurons with the first leakage mechanism can be described as follows.
Denote the initial list of membrane potentials  $U_0^{N,u}=u \in \st$. 
The list of membrane potentials $(U_t^{N,u})_{t\in [0,+\infty)}$ evolves as a Markov jump process taking values in the set $\st$
and with infinitesimal generator $\mathcal{G}$ defined as follows
\begin{equation*} 
\mathcal{G}f(u)= \sum_{b\in \N}
e^{  u(b)}\mathbf{1}\{u(b)>0\}\left[f(\pi^{b,*}(u))-f(u)\right]+\sum_{b\in \N}\left[f(\pi^{b,\dagger}(u))-f(u)\right],
\end{equation*}
for any  bounded function $f:\st \to \mathbb{R}$.

The time evolution of the system of neurons with the second leakage mechanism can be described as follows.
The list of membrane potentials $(\hat{U}_t^{N,u})_{t\in [0,+\infty)}$ evolves as a Markov jump process taking values in the set $\st$ with initial list $u \in \st$
and with infinitesimal generator $\hat{\mathcal{G}}$ defined as follows
\begin{equation*} 
\hat{\mathcal{G}}f(u)= \sum_{b\in \N}
e^{u(b)}\mathbf{1}\{u(b)>0\}\left[f(\pi^{b,*}(u))-f(u)\right]+\sum_{b\in \N}\left[f(\hat{\pi}^{b,\dagger}(u))-f(u)\right],
\end{equation*}
for any  bounded function $f:\st \to \mathbb{R}$.

Observe that the null list $\nlist \in \st$,  defined as
$$
\nlist(a)=0, \text{ for any } a \in \N
$$
is a trap for both processes.
The goal of this article is to study the time the process takes to get trapped and its behavior before get trapped. 

To state our main results, we need to introduce some notation.
Let 
$$
\tau^{N,u}= \inf\{t>0: U_t^{N,u} = \nlist\}
$$
and define $\mathcal{N}^{N,u}$ as the number of spikes and leakages of membrane potential of the process, namely
$$
\mathcal{N}^{N,u}=\left\lvert\left\{s \in (0,\tau^{N,u}]:U_s^{N,u}\neq \lim_{t \to s^-}U_t^{N,u}\right\}\right\rvert.
$$
Analogously, let
$$
\hat{\tau}^{N,u}= \inf\{t>0: \hat{U}_t^{N,u} = \nlist\}
$$
and
$$
\hat{\mathcal{N}}
^{N,u}=\left\lvert\left\{s \in (0,\hat{\tau}^{N,u}]:\hat{U}_s^{N,u}\neq \lim_{t \to s^-}\hat{U}_t^{N,u}\right\}\right\rvert.
$$
We consider also the set
$$
S_N^{(0)}=\{u \in \st: \lvert\{a \in \N: u(a)>0\}\rvert \geq \lfloor N^{1/2} \rfloor \}
$$
and the set
$$
\W=
\left\{u \in \st: I_N\subset \{u(a): a \in \N\}
,\ \bigcap_{a \in \N} \bigcap_{b \neq a} \{u(a) \neq u(b)\}
\right\},
$$
where $I_N=\left\{1,\ldots, N - \lfloor N^{1/2} \rfloor \right\}$.

$S_N^{(0)}$ and $\W$ are the sets described in the informal presentation of Theorems \ref{massconcentration} and \ref{massconcentration2}. $S_N^{(0)}$ is the set of configurations in which a sufficiently large set of neurons have strictly positive membrane potential. This sufficiently large set of neurons has size greater or equal $\lfloor N^{1/2} \rfloor$. $\W$ is the set in which the process is with probability $1$ as $N\to +\infty$ before it gets trapped.

Informally, starting from $S_N^{(0)}$, with probability $1$ as $N\to +\infty$ the process will have a sequence of spikes and the difference between the membrane potential of the neuron which spikes and the greatest membrane potential at each time will be small. More specifically, this difference will be at most $N^{1/2}$. This sequence of spikes will lead the process to a situation in which all neurons have different membrane potentials and the membrane potential of the neurons forms a part of a ladder. This means that there is one neuron with membrane potential equal to $1$, one neuron with membrane potential equal to $2$ and so on until one neuron with membrane potential equal to $N-\lfloor N^{1/2} \rfloor$. This is exactly the configurations contained in $\W$.

We can now state our main results.

\begin{theorem} \label{teo: extinction}
For any $N \geq 2$ and for any initial list $u \in \st$, it follows that
$$
\P(\mathcal{N}^{N,u}<+\infty)=1
$$
and
$$
\P(\tau^{N,u}<+\infty)=1.
$$
\end{theorem}

\begin{theorem} \label{massconcentration}
For any $t>0$, it follows that
$$
\inf_{u \in S_N^{(0)}}\P\left(U_{t }^{N,u}\in \W \ \vert \  \tau^{N,u}>t \right) \to 1, \text{ as } N \to +\infty.
$$
\end{theorem}

\begin{theorem} \label{metastable}
For any sequence $(u_N \in S_N^{(0)}: N \geq 2)$,
$$
\frac{\tau^{N,u_N}}{\mathbb{E}[\tau^{N,u_N}]}\to \text{Exp}(1) \text{ in distribution, as } N \to +\infty,
$$
where $\text{Exp}(1)$ denotes a mean $1$ exponential distributed random variable.

\end{theorem}

\begin{theorem} \label{teo: extinction2}
For any $N \geq 2$ and for any initial list $u \in \st$, it follows that
$$
\P(\hat{\mathcal{N}}^{N,u}<+\infty)=1
$$
and
$$
\P(\hat{\tau}^{N,u}<+\infty)=1.
$$
\end{theorem}
\begin{theorem} \label{massconcentration2}
For any $t>0$, it follows that
$$
\inf_{u \in S_N^{(0)}}\P\left(\hat{U}_{t }^{N,u}\in \W \ \vert \  \hat{\tau}^{N,u}>t \right) \to 1, \text{ as } N \to +\infty.
$$
\end{theorem}
\begin{theorem} \label{metastable2}
For any sequence $(u_N \in S_N^{(0)}: N \geq 2)$,
$$
\frac{\hat{\tau}^{N,u_N}}{\mathbb{E}[\hat{\tau}^{N,u_N}]}\to \text{Exp}(1) \text{ in distribution, as } N \to +\infty,
$$
where $\text{Exp}(1)$ denotes a mean $1$ exponential distributed random variable.

\end{theorem}


In Sections \ref{sec:teoextinction}, \ref{sec:auxiliary}, \ref{sec:teo1} and \ref{sec:teo2} we prove Theorems \ref{teo: extinction}, \ref{massconcentration} and \ref{metastable} concerning the  process $(U_t^{N,u})_{t\in [0,+\infty)}$. In Section \ref{sec:hatprocess} we show how the proofs presented before can be modified in order to prove Theorems \ref{teo: extinction2}, \ref{massconcentration2} and \ref{metastable2} concerning the  process $(\hat{U}_t^{N,u})_{t\in [0,+\infty)}$.




To prove our results it is convenient to extend the notation introduced before.

\vspace{0.2cm}

\noindent \textbf{Extra notation}
\begin{itemize}
\item
Let $T_0=0$ and for $n =1,\ldots, \mathcal{N}^{N,u}$, let $T_n$ denote  the successive jumping times of the process $(U_t^{N,u})_{t\in [0,+\infty)}$, namely 
$$
T_n= \inf\left\{t>T_{n-1}:U_t^{N,u} \neq U^{N,u}_{T_{n-1}}\right\}.
$$

\item For $n =1,\ldots, \mathcal{N}^{N,u}$, we define $A_n \in \N$ and $O_n \in \{*,\dagger\}$ as the pair such that
$$
U_{T_n}^{N,u}=\pi^{A_n,O_n}\left(U^{N,u}_{T_{n-1}}\right).
$$



\item The leakage times are defined as $T_0^{\dagger}=0$ and for $n\geq 1$,
    $$
   T_n^{\dagger}=\inf\{T_{m}>T_{n-1}^{\dagger}:O_{m}=\dagger\}.
    $$
    
\item The spiking times are defined as $T_0^{*}=0$ and for $n\geq 1$,
    $$
    T_n^{*}=\inf\{T_{m}>T_{n-1}^{*}:O_{m}=*\}.
    $$
    
\item For any time interval $I \subset [0,+\infty)$, the counting measures indicating the number of leakage times  and spiking times that occurred during the time interval $I$ are defined as
$$
Z^{\dagger}(I)= \sum_{m=1}^{+\infty}\mathbf{1}\{T_m^{\dagger} \in I\} \ \ \ \ \text{ and } \ \ \ \ 
Z^{*}(I)= \sum_{m=1}^{+\infty}\mathbf{1}\{T_m^{*} \in I\}.
$$

\item 
For any $u \in \st$, we define $a_1^u, \ldots, a_N^u \in \N$ in the following way
$$
a_1^u\in \text{argmin}\{u(a):a \in \N\},
$$
$$
a_2^u \in \text{argmin}\{u(a):a \in \N\setminus\{a_1^u\}\},
$$
$$
...
$$
$$
a_N^u \in \text{argmin}\{u(a):a \in \N\setminus\{a_1^u,a_2^u,...,a_{N-1}^u\}\}.
$$
To avoid ambiguity, we use the following convention: if $u(a_j^u)=u(a_{j+1}^u)$, then $a_j^u <a_{j+1}^u$. 

$a_1^u$ is the neuron with smallest membrane potential in configuration $u$ (it will always satisfy $u(a_1^u)=0$) and so on until $a_N^u$, the neuron with greatest membrane potential in configuration $u$. Since it is possible to have neurons with the same membrane potential in configuration $u$, we need to define the convention above.

\item The set of ladder lists is defined as
$$
\ladder=\big\{u \in \st: \{u(a): a \in \N\}=\{0,1,...,N-1\}\big\}.
$$

 
 

\item Let $\sigma: \N \to \N$ be a bijective map. For any $u\in \st$, the permuted list $\sigma(u)\in \st$ is defined as
$$
\sigma (u)(a)=u(\sigma(a))  \text{, for all } a \in \N.
$$

\item For any $\lambda>0$,  $\xi^{\{\lambda\}}$ and $(\xi_j^{\{\lambda\}}: j=1,2,\ldots)$  will always be, respectively, a random variable exponentially distributed with mean $\lambda^{-1}$ and a sequence of independent  random variables exponentially distributed with mean $\lambda^{-1}$.

\end{itemize}


\section{Proof of Theorem \ref{teo: extinction}} \label{sec:teoextinction}

In this section we will prove Theorem \ref{teo: extinction}. To prove Theorem \ref{teo: extinction}, we first show that starting from any non null list, the process has a positive and bounded probability to reach the set of ladder lists after $N-1$ jumps of the process. This follows from the fact that a ladder list is obtained when for $N-1$ consecutive instants, there is the spiking of a neuron in the process and the neuron that spiked is the neuron with maximum membrane potential at each instant.
This is the content of Lemma \ref{ladderreaching}. 
With this Lemma we are able to prove Theorem \ref{teo: extinction}.

\begin{lemma} \label{ladderreaching}
For any $u \in \st \setminus \{\nlist\}$, it follows that
$$
\P(U_{T_{N-1}}^{N,u} \in \ladder) \geq \left(\frac{1}{2(N-1)}\right)^{N-1}.
$$
\end{lemma}
\begin{proof}
For any initial list $u \in \st \setminus \{\nlist\}$, the occurrence of the event $\{A_1=a_N^{u},O_1=*\}$ implies that $U_{T_1}^{N,u} \in \ladder$ in the case $N=2$, and implies that
$$
a_1^{U_{T_1}^{N,u}}=0, \ a_2^{U_{T_1}^{N,u}}=1,\ a_j^{U_{T_1}^{N,u}}\geq 1 \text{, for } j=3,\ldots, N,
$$
in the case $N\geq 3$.
As a consequence, the occurrence of the event $$
\left\{A_1=a_N^{u},O_1=*, A_2=a_N^{U_{T_1}^{N,u}},O_2=*\right\}
$$ 
implies that implies that $U_{T_2}^{N,u} \in \ladder$ in the case $N=3$, and it implies that
$$
a_1^{U_{T_2}^{N,u}}=0, \ a_2^{U_{T_2}^{N,u}}=1,\ a_3^{U_{T_2}^{N,u}}=2, \ a_j^{U_{T_2}^{N,u}}\geq 2 \text{, for } j=4,\ldots, N.
$$ 
in the case $N\geq 4$. Iterating this, we conclude that
the occurrence of the event 
$$
\bigcap_{j=1}^{N-1}\{A_j=a_N^{U_{T_{j-1}}^{N,u}}, O_j=*\}
$$ 
implies that $U_{T_{N-1}}^{N,u} \in \ladder$.
Therefore,
\begin{equation} \label{eqlemma1}
\P(U_{T_{N-1}}^{N,u} \in \ladder) \geq \P\left(\bigcap_{j=1}^{N-1}\{A_j=a_N^{U_{T_{j-1}}^{N,u}}, O_j=*\}\right).
\end{equation}

The smallest value for
$$
\P(A_1=a_N^{u},O_1=*)
$$
is obtained for any initial list $u$ in which all neurons, except one, have membrane potential equal $1$. This implies that
$$
\inf\left\{\P(A_1=a_N^{v},O_1=*\ \vert \ U_0^{N,v}=v):v \in \st \setminus\{\nlist\}\right\}\geq \frac{1}{2(N-1)}.
$$
We conclude the proof by using Markov property and applying this lower bound $N-1$ times in Equation \eqref{eqlemma1}.
\end{proof}

\begin{proof}
We will now prove Theorem \ref{teo: extinction}.

For any $N\geq 2$ and for any $u,u' \in \st \setminus \{\nlist\}$, we have that
\begin{align*}
    &\P\Big(
    \mathcal{N}^{N,u}\leq n+2(N-1)\ \Big\vert\ U^{N,u}_{T_n}=u'\Big) \geq \\ 
    &\P\left(U_{T_{n+N-1}}^{N,u} \in \ladder \ \Big\vert \ U_{T_n}^{N,u}=u'\right)\P\left(\mathcal{N}^{N,u}\leq n+2(N-1)\ \Big\vert\ U_{T_{n+N-1}}^{N,u} \in \ladder\right).
\end{align*}
Using together the Markov property and Lemma \ref{ladderreaching}, we get 
$$
\P\left(U_{T_{n+N-1}}^{N,u} \in \ladder \ \Big\vert \ U_{T_n}^{N,u}=u'\right) \geq [2(N-1)]^{-(N-1)}.
$$
The invariance by permutation of the process implies that
$$
\P\left(\mathcal{N}^{N,u}\leq m+(N-1) \ \Big\vert\ U_{T_{m}}^{N,u}=l\right)=\P\left(\mathcal{N}^{N,u}\leq m+(N-1) \ \Big\vert\ U_{T_{m}}^{N,u}=l'\right),
$$
for any $l,l' \in \ladder$ and for any $m\geq 1$. Calling
$$
\epsilon'=\P\left(\mathcal{N}^{N,u}\leq n+2(N-1) \ \Big\vert\ U_{T_{n+N-1}}^{N,u}=l\right),
$$
we conclude that
$$
\P\left(\mathcal{N}^{N,u}\leq n+2(N-1) \ \Big\vert \ U_{T_{n+N-1}}^{N,u} \in \ladder\right) = \epsilon' >0. 
$$

Therefore, for any $u' \in \st$,
$$
\P\Big(\mathcal{N}^{N,u}\leq n+2(N-1)\ \Big\vert\ U^{N,u}_{T_n}=u'\Big) \geq [2(N-1)]^{-(N-1)}\epsilon'.
$$
The last inequality implies that for any $n\geq 1$,
$$
\P(\mathcal{N}^{N,u} \geq n) \leq \P\big(2(N-1) \times \text{Geom}\big([2(N-1)]^{-(N-1)}\epsilon'\big) \geq n\big),
$$
where $\text{Geom}(r)$ denotes a random variable with geometric distribution assuming values in $\{1,2,\ldots\}$ and with mean $1/r$. 
This implies that $\P(\mathcal{N}^{N,u}< +\infty)=1$, concluding the first part of the proof.

The jump rate of the process $(U_t^{N,u})_{t\in [0,+\infty)}$ satisfies
$$
\sum_{a \in \N}\mathbf{1}\{u'(a)>0\}(e^{ u'(a)}+1) \geq e+1,
$$
for any $u' \in \st \setminus \{\nlist\}$. Putting all this together we conclude that $\P(\tau^{N,u}< +\infty)=1$.

\end{proof}

\section{A coupling construction} \label{sec:auxiliary}

In this section we will prove the following proposition.

\begin{proposition} \label{unifconvergence}
The following holds
$$
\lim_{N \to +\infty}\sup_{t\geq 0}\sup_{w,w' \in \W}\left\lvert\P(\tau^{N,w}>t)-\P(\tau^{N,w'}>t) \right\rvert=0.
$$
\end{proposition}

To prove Proposition \ref{unifconvergence}, we need to introduce a coupling construction of the processes $(U_t^{N,u'})_{t\in [0,+\infty)}$ and $(U_t^{N,v'})_{t\in [0,+\infty)}$ starting from two different lists $u',v' \in \st$. 

We want to describe the time evolution of  $(U_t^{N,u'},U_t^{N,v'})_{t\in [0,+\infty)}$.
To do this, for any index $j \in \{1,\ldots,N\}$, we define the maps $\pi^{j,\min}$, $\pi^{j,\max}$ and $\pi^{j,\dagger}$ on $\st^2$
as follows. For any $(u,v)\in \st^2$, 
$$
\pi^{j,\min}(u,v)=(\pi^{a_j^u,*}(u), \pi^{a_j^v,*}(v)),
$$
$$
\pi^{j,\max}(u,v)=
\begin{cases}
(\pi^{a_j^u,*}(u), v) &\text{, if } u(a_j^u)>v(a_j^v), \\
(u, \pi^{a_j^v,*}(v)) &\text{, if } v(a_j^v)>u(a_j^u),
\end{cases}
$$
$$
\pi^{j,\dagger}(u,v)=(\pi^{a_j^u,\dagger}(u), \pi^{a_j^v,\dagger}(v)).
$$

The map $\pi^{j,\min}(u,v)$ represents the simultaneous effect of a spike of neuron $a_j^u$ in the system $(U_t^{N,u'})_{t\geq0}$ and a spike of neuron $a_j^v$ in the system $(U_t^{N,v'})_{t\in [0,+\infty)}$. 

The map $\pi^{j,\max}(u,v)$ represents the effect of either a spike of neuron $a_j^u$ in the system $(U_t^{N,u'})_{t\geq0}$ in the case in which $u(a_j^u)>v(a_j^v)$, or a spike of neuron $a_j^v$ in the system $(U_t^{N,v'})_{t\in [0,+\infty)}$ in the case in which $v(a_j^v)>u(a_j^u)$.


The map $\pi^{j,\dagger}(u,v)$ represents the  simultaneous leakage effect on the membrane potential of neuron $a_j^u$ in the system $(U_t^{N,u'})_{t\geq0}$ and on the membrane potential of neuron $a_j^v$ in the system $(U_t^{N,v'})_{t\in [0,+\infty)}$. 


The pair of lists of membrane potentials $(U_t^{N,u'},U_t^{N,v'})_{t\in [0,+\infty)}$ evolves as a Markov jump process taking values in the set $\st^2$
and with infinitesimal generator $\mathcal{G}_C$ defined as follows
$$
\mathcal{G}_Cf(u,v)= \sum_{j=1}^N
e^{\vert u(a_j^u)-v(a_j^v)\vert}\mathbf{1}\{u(a_j^u)\neq v(a_j^v)\}\left[f(\pi^{j,\max}(u,v))-f(u,v)\right]+
$$
$$
\sum_{j=1}^N
e^{\min\{u(a_j^u),v(a_j^v)\}}\mathbf{1}\{\min\{u(a_j^u),v(a_j^v)\}>0\}\left[f(\pi^{j,\min}(u,v))-f(u,v)\right]+
$$
$$
\sum_{j=1}^N\left[f(\pi^{j,\dagger}(u,v))-f(u,v)\right],
$$
for any  bounded function $f:\st^2 \to \mathbb{R}$.

\vspace{0.2cm}
 
For the coupling construction we introduce some extra notation.

\vspace{0.2cm}

\noindent \textbf{Extra notation - coupling construction}
\begin{itemize}
\item Define
$$
\tau^{N}(u,v)=\inf\{s>0:(U_s^{N,u},U_s^{N,v})=(\nlist,\nlist)\}.
$$
\item Define $\mathcal{N}^{N}(u,v)$ as the number of spikes and leakages of membrane potential of the coupling process, namely
$$
\mathcal{N}^{N}(u,v)=\left\lvert\left\{
s>0
:(U_s^{N,u},U_s^{N,v})\neq \left(\lim_{t \to s^-} U_t^{N,u}, \lim_{t \to s^-} U_t^{N,v}\right)\right\}\right\rvert.
$$
\item
Let $T_0(u,v)=0$ and for $n =1,\ldots, \mathcal{N}^{N}(u,v)$ denote $T_n(u,v)$ the successive jumping times of the process $(U_t^{N,u},U_t^{N,v})_{t\in [0,+\infty)}$, namely 
$$
T_n(u,v)= \inf\left\{t>T_{n-1}(u,v):(U_t^{N,u},U_t^{N,v}) \neq \big(U^{N,u}_{T_{n-1}(u,v)},U^{N,v}_{T_{n-1}(u,v)}\big)\right\}.
$$

\item For each  $n =1,\ldots, \mathcal{N}^{N}(u,v)$, we define $J_n(u,v) \in \{1,\ldots,N\}$ and $K_n(u,v) \in \{\min,\max,\dagger\}$ as the pair such that
$$
(U_{T_n(u,v)}^{N,u},U_{T_n(u,v)}^{N,v})=\pi^{J_n(u,v),K_n(u,v)}\left(U^{N,u}_{T_{n-1}(u,v)},U^{N,v}_{T_{n-1}(u,v)}\right).
$$
The pair $(J_n(u,v),K_n(u,v))$ is exactly the index of the map used to transform the list of membrane potentials at time $T_n(u,v)$
describing the change in the coupling process at this time.

\item For any $j\geq 1$, we define the event
$$
E_j(u,v)=\bigcap_{n=2(j-1)\lceil N^{1/2} \rceil+1}^{2j \lceil N^{1/2} \rceil}\left\{J_n(u,v)=N, K_n(u,v) \neq \dagger\right\}.
$$
$E_j(u,v)$ is the event in which the neuron with greatest membrane potential spikes (either simultaneously on both evolutions or not) at the jump times $T_n(u,v)$, for $n=2(j-1)\lceil N^{1/2} \rceil+1, \ldots, 2j \lceil N^{1/2} \rceil$.

\item The number of jumping times of the process $(U_t^{N,u},U_t^{N,v})_{t\in [0,+\infty)}$ until the first leakage time is defined as
$$
\mathcal{N}_{\dagger}^N(u,v)=\inf\{
n:K_n(u,v)=\dagger\}.
$$

\item The number of jumping times of the process $(U_t^{N,u},U_t^{N,v})_{t\in [0,+\infty)}$ until the coupling time is defined as
$$
\mathcal{N}^N_C(u,v)=\inf\left\{
n: 
\exists\  \sigma: \N \to \N \text{ bijective s.t. } U_{T_n(u,v)}^{N,u}=\sigma\left(U_{T_n(u,v)}^{N,v}\right)\right\}.
$$
\end{itemize}

\begin{remark} \label{remark0}
There exists a bijective map $\sigma: \N \to \N $ such that 
$$
U_s^{N,u}=\sigma\left(U_s^{N,v}\right) \text{, for all $s \geq T_{\mathcal{N}^N_C(u,v)}(u,v)$.}
$$
Moreover, if there exists $t\geq 0$ such that $U_t^{N,u} \in \ladder$ and $U_t^{N,v} \in \ladder$ , then $t \geq T_{\mathcal{N}^N_C(u,v)}(u,v)$.
\end{remark}



The proof of Proposition \ref{unifconvergence} is based on the three following lemmas about the coupling construction. In Lemma \ref{lemmaE1} we prove that starting from two lists $w, w' \in \W$, if the event $E_1(u,v)$ occurs then the number of jumps before the coupling time is smaller or equal $2\lceil N^{1/2} \rceil$. This is based on the fact that a sequence of spikes of neurons with greatest membrane potential lead the process to a ladder list and the fact that the coupling time occurs before or at the same time in which both processes reach simultaneously on the set of ladder lists. In Lemmas \ref{lemmageombound} and \ref{couplingtime} we obtain a bound for the number of jumping times of the process before the coupling time and shows that with probability $1$ as $N\to +\infty$, the coupling time occurs before the process gets trapped. As a consequence, in Corollary \ref{coro} we prove that the coupling time occurs instantaneously as $N\to +\infty$. Putting all this together we are able to prove Proposition \ref{unifconvergence}.

\begin{lemma} \label{lemmaE1} For any lists $w, w' \in \W$, if the event $E_1(w,w')$ occurs, then
$$
\mathcal{N}^N_C(w,w') \leq 2\lceil N^{1/2} \rceil.
$$
\end{lemma}
\begin{proof}
The occurrence of the event
$$
E_1(w,w')=\bigcap_{n=1}^{2 \lceil N^{1/2} \rceil}\left\{J_n(w,w')=N, K_n(w,w') \neq \dagger\right\}
$$
implies that in the first $2 \lceil N^{1/2} \rceil$ steps of the coupling construction there are neurons spiking and at each step, the neuron that spikes is the neuron with greatest membrane potential.

For the first step, denoting $u_1=U_{T_1(w,w')}^{N,w}$ and $u'_1=U_{T_1(w,w')}^{N,w'}$, we have two possible cases:
\begin{itemize}
    \item If $J_1(w,w')=N$ and $K_1(w,w')=\min$, then
    $$
    \left\{1,\ldots, N - \lfloor N^{1/2} \rfloor +1 \right\} \subset \{u_1(a): a \in \N\}, \ \ \bigcap_{a \in \N} \bigcap_{b \neq a} \{u_1(a) \neq u_1(b)\}
    $$
    and
     $$
    \left\{1,\ldots, N - \lfloor N^{1/2} \rfloor +1 \right\} \subset \{u'_1(a): a \in \N\}, \ \ \bigcap_{a \in \N} \bigcap_{b \neq a} \{u'_1(a) \neq u'_1(b)\}.
    $$
    
     \item If $J_1(w,w')=N$ and $K_1(w,w')=\max$,  then either $u'_1=w'$ and
      $$
    \left\{1,\ldots, N - \lfloor N^{1/2} \rfloor +1 \right\} \subset \{u_1(a): a \in \N\}, \ \ \bigcap_{a \in \N} \bigcap_{b \neq a} \{u_1(a) \neq u_1(b)\}
    $$
    in the case $u_1(a_N^{u_1}) > u'_1(a_N^{u'_1})$, or $u_1=w$ and
     $$
    \left\{1,\ldots, N - \lfloor N^{1/2} \rfloor +1 \right\} \subset \{u'_1(a): a \in \N\}
, \ \ \bigcap_{a \in \N} \bigcap_{b \neq a} \{u'_1(a) \neq u'_1(b)\}
    $$
    in the case $u'_1(a_N^{u'_1}) > u_1(a_N^{u_1})$.
\end{itemize}

Iterating this, we conclude that if the event $E_1(w,w')$ occurs, then
\begin{equation} \label{ladderreach}
U_{T_{2\lceil N^{1/2} \rceil}(w,w')}^{N,w} \in \ladder \ \ \ \ \text{ and } \ \ \ \ U_{T_{2\lceil N^{1/2} \rceil}(w,w')}^{N,w'}\in \ladder.
\end{equation}
By Remark \ref{remark0},  \eqref{ladderreach} implies that 
$\mathcal{N}^N_C(w,w') \leq 2 \lceil N^{1/2} \rceil$.
\end{proof}

\begin{lemma}\label{lemmageombound} For any $n\geq 1$ and for any $w,w' \in \W$,
$$
\P\big(\mathcal{N}^N_C(w,w')\leq 2n \lceil N^{1/2} \rceil < \mathcal{N}_{\dagger}^N(w,w')\big) \geq
$$
$$
\P\left(\text{Geom}\left(\zeta^{2\lceil N^{1/2} \rceil}\right) \leq n\right)\left(\frac{e^{\lfloor N^{1/2} \rfloor}}{e^{\lfloor N^{1/2} \rfloor}+2(N-1)}\right)^{2n \lceil N^{1/2} \rceil},
$$
where
$
\zeta=1-e^{-1}
$
and $\text{Geom}\left(\zeta^{2\lceil N^{1/2} \rceil}\right)$ is a random variable with geometric distribution assuming values in $\{1, 2,\ldots\}$ and with mean $1/\zeta^{2\lceil N^{1/2} \rceil}$.
\end{lemma}
\begin{proof} 
To simplify the presentation of the proof, for a fixed pair of lists $w,w' \in \W$ and for any $m\geq 1$, we will use the shorthand notation $J_m$, $K_m$ and $E_m$ instead
of $J_m(w,w')$, $K_m(w,w')$ and $E_m(w,w')$, respectively.

For any $n\geq 1$, the occurrence of the event
$$
\bigcap_{m=1}^{2n \lceil N^{1/2} \rceil}\left\{J_m \in \{N-\lfloor N^{1/2} \rfloor +1,\ldots, N\}, K_m\neq \dagger \right\}
$$
implies that $U_m^{N,w} \in \W$ and $U_m^{N,w'} \in \W$, for all $m=1,\ldots, 2n \lceil N^{1/2} \rceil$. This implies that
for any $w,w' \in \W$,
$$
\P\big(\mathcal{N}^N_C(w,w')\leq 2n \lceil N^{1/2} \rceil < \mathcal{N}_{\dagger}^N(w,w')\big) \geq
$$
$$
\P\left(\bigcup_{m=1}^{n}E_m,\bigcap_{m=1}^{2n \lceil N^{1/2} \rceil}\left\{J_m \in \{N-\lfloor N^{1/2} \rfloor +1,\ldots, N\}, K_m\neq \dagger \right\}\right).
$$

For any lists $u, v \in \W$ 
we have that
\begin{equation} \label{eqprop2lemma2}
\P(J_1(u,v)=N, K_1(u,v) \neq \dagger)= \frac{e^{ \max\{u(a_N^{u}),v(a_N^{v})\} }}{\displaystyle \sum_{j=2}^{N} e^{ \max\{u(a_j^{u}),v(a_j^{v})\} }}\P(K_1(u,v) \neq \dagger).
\end{equation}
The left term of the right-hand side in Equation \eqref{eqprop2lemma2} is bounded below by 
$$
\frac{e^{ (N-1)}}{\displaystyle\sum_{j=1}^{N-1} e^{ j}} \geq \zeta.
$$
Therefore,
$$
\P\left(E_1, \bigcap_{m=1}^{2\lceil N^{1/2} \rceil}\left\{J_m \in \{N-\lfloor N^{1/2} \rfloor +1,\ldots, N\}, K_m\neq \dagger\right\} \right) \geq 
$$
$$
\zeta^{2\lceil N^{1/2} \rceil}\P\left( \bigcap_{m=1}^{2\lceil N^{1/2} \rceil}\left\{J_m \in \{N-\lfloor N^{1/2} \rfloor +1,\ldots, N\}, K_m\neq \dagger\right\} \right) ,
$$
and more generally, for any $n=1,2,\ldots$,
$$
\P\left(\bigcup_{m=1}^nE_m, \bigcap_{m=1}^{2n \lceil N^{1/2} \rceil}\left\{J_m \in \{N-\lfloor N^{1/2} \rfloor +1,\ldots, N\}, K_m\neq \dagger\right\}\right) \geq
$$
$$
\left(1-\big(1-\zeta^{2\lceil N^{1/2} \rceil}\big)^n\right)\P\left(\bigcap_{m=1}^{2n \lceil N^{1/2} \rceil}\left\{J_m \in \{N-\lfloor N^{1/2} \rfloor +1,\ldots, N\}, K_m\neq \dagger\right\}\right).
$$

To conclude the proof, note that for any lists $u, v \in \W$, we have $\max\{u(a_N^{u}),v(a_N^{v})\} \geq N-1$ and
$u\big(a_{N-\lfloor N^{1/2} \rfloor}^{u}\big)=v\big(a_{N-\lfloor N^{1/2} \rfloor}^{v}\big)=N-\lfloor N^{1/2} \rfloor-1$.
This implies that
$$
\P\left(\bigcap_{m=1}^{2n \lceil N^{1/2} \rceil}\left\{J_m \in \{N-\lfloor N^{1/2} \rfloor +1,\ldots, N\}, K_m\neq \dagger\right\}\right) \geq 
$$
$$
\left(\frac{e^{\lfloor N^{1/2} \rfloor}}{e^{\lfloor N^{1/2} \rfloor}+2(N-1)}\right)^{2n \lceil N^{1/2} \rceil}.
$$
\end{proof}

\begin{lemma} \label{couplingtime} The following holds 
$$
\inf_{w,w' \in \W}\P\big(\mathcal{N}^N_C(w,w')< e^{N^{1/2}}N^{-2} < \mathcal{N}_{\dagger}^N(w,w')\big) \to 1, \text{ as } N \to +\infty.
$$
\end{lemma}
\begin{proof}
For any $w,w' \in \W$, taking $n=\lfloor e^{N^{1/2}}N^{-2}\rfloor / (2\lceil N^{1/2} \rceil) $ 
in Lemma \ref{lemmageombound}, we have that
\begin{equation} \label{propbound2}
\P\big(\mathcal{N}^N_C(w,w')< \lfloor e^{N^{1/2}}N^{-2} \rfloor < \mathcal{N}_{\dagger}^N(w,w')\big) \geq
\end{equation}
$$
\left(1-\left(1-\zeta^{2\lceil N^{1/2} \rceil}\right)^{\lfloor e^{N^{1/2}}N^{-2}\rfloor / (2\lceil N^{1/2} \rceil) }\right)\left(\frac{e^{\lfloor N^{1/2} \rfloor}}{e^{\lfloor N^{1/2} \rfloor}+2N}\right)^{\lfloor e^{N^{1/2}}N^{-2}\rfloor} \to 1, 
$$
as $N \to +\infty.$
To finish the proof, just note that the bound of Equation  \ref{propbound2} does not depend on the initial lists $w,w' \in \W$ and take $n=\lceil e^{N^{1/2}}N^{-2}\rceil / (2\lceil N^{1/2} \rceil)$ in Lemma \ref{lemmageombound}.
\end{proof}

\begin{corollary} \label{coro}
The following holds
    $$
\sup_{w, w' \in \W}\P(T_{\mathcal{N}^N_C(w,w')}(w,w')>e^{-(N-N^{1/2})}) \to 0, \text{ as } N \to +\infty.
$$
\end{corollary}
\begin{proof} For any $w, w' \in \W$ and for any $t>0$, 
$$
\P(T_{\mathcal{N}^N_C(w,w')}(w,w')>t)\leq
$$
$$
\P(T_{\mathcal{N}^N_C(w,w')}(w,w')>t, \mathcal{N}^N_C(w,w')< e^{N^{1/2}}N^{-2} < \mathcal{N}_{\dagger}^N(w,w'))+
$$
$$
\P(\{\mathcal{N}^N_C(w,w')< e^{N^{1/2}}N^{-2} < \mathcal{N}_{\dagger}^N(w,w')\}^c).
$$
Lemma \ref{couplingtime} implies that
$$
\sup_{w, w' \in \W}\P(\{\mathcal{N}^N_C(w,w')< e^{N^{1/2}}N^{-2} < \mathcal{N}_{\dagger}^N(w,w')\}^c)\to 0, \text{ as } N \to +\infty.
$$
Moreover,
$$
\P(T_{\mathcal{N}^N_C(w,w')}(w,w')>t, \mathcal{N}^N_C(w,w')< e^{N^{1/2}}N^{-2} < \mathcal{N}_{\dagger}^N(w,w'))\leq
$$
\begin{equation} \label{eqcoro}
\P(T_{\lfloor e^{N^{1/2}}N^{-2}\rfloor }(w,w')>t, \mathcal{N}^N_C(w,w')< e^{N^{1/2}}N^{-2} < \mathcal{N}_{\dagger}^N(w,w')).
\end{equation}
For any initial lists $w,w'$ and for any $s>0$, if the event $\mathcal{N}_{\dagger}^N(w,w') > e^{N^{1/2}}N^{-2}$ occurs, then
$$
\P(T_{j}(w,w')-T_{j-1}(w,w')>s)\leq \P(\xi^{e^{(N-1)}}>s), \text{ for any } j=1,\ldots, \lfloor e^{N^{1/2}}N^{-2} \rfloor.
$$
Therefore, taking $t=e^{-(N-N^{1/2})}$ the right-hand side of Equation \eqref{eqcoro} is bounded above by
\begin{equation} \label{eqcoro2}
\P\left(\sum_{j=1}^{\lfloor e^{N^{1/2}}N^{-2} \rfloor}\xi_j^{\{e^{ (N-1)}\}}>e^{-(N-N^{1/2})}\right) \to 0, \text{ as } N \to +\infty.
\end{equation}
We conclude the proof by putting Equations \eqref{eqcoro} and \eqref{eqcoro2} together and noting that the bound on \eqref{eqcoro2} 
does not depend on the lists $w, w' \in \W$.

\end{proof}

\begin{remark} \label{remarkladderreach} By Remark \ref{remark0} and Equation \eqref{ladderreach}, we can replace $\mathcal{N}^N_C(w,w')$ by 
$$
\inf\{s>0: \{U_s^{N,w} \in \ladder\} \cap \{U_s^{N,w'} \in \ladder\} \}
$$ 
in Lemmas \ref{lemmaE1}, \ref{lemmageombound} and \ref{couplingtime}.
This implies that
$$
\sup_{w \in \W}\P(\inf\{s>0: U_s^{N,w} \in \ladder \}>e^{-(N-N^{1/2})}) \to 0, \text{ as } N \to +\infty.
$$
\end{remark}

\begin{proof} We have now all the ingredients to prove Proposition \ref{unifconvergence}.

For any $t>0$ and for any $w,w' \in \W$,
$$
\P(\tau^{N,w}>t)\leq \P(\tau^{N,w}>t, \mathcal{N}^N_C(w,w') < \mathcal{N}_{\dagger}^N(w,w'))+\P(\mathcal{N}^N_C(w,w') > \mathcal{N}_{\dagger}^N(w,w'))
$$
Now, note that
$$
\P(\tau^{N,w}>t, \mathcal{N}^N_C(w,w') < \mathcal{N}_{\dagger}^N(w,w'))=
\P(\tau^{N,w'}>t, \mathcal{N}^N_C(w,w') < \mathcal{N}_{\dagger}^N(w,w')).
$$
This implies that
$$
\P(\tau^{N,w}>t)-\P(\tau^{N,w'}>t) \leq \P(\mathcal{N}^N_C(w,w') > \mathcal{N}_{\dagger}^N(w,w')).
$$
Analogously,
$$
\P(\tau^{N,w'}>t)-\P(\tau^{N,w}>t) \leq \P(\mathcal{N}^N_C(w,w') > \mathcal{N}_{\dagger}^N(w,w')),
$$
and therefore,
$$
\vert\P(\tau^{N,w}>t)-\P(\tau^{N,w'}>t)\vert \leq \P(\mathcal{N}^N_C(w,w') > \mathcal{N}_{\dagger}^N(w,w')).
$$

By Lemma \ref{couplingtime}, we conclude that
$$
\sup_{t\geq 0}\ \sup_{w,w'\in \W}\vert\P(\tau^{N,w}>t)-\P(\tau^{N,w'}>t)\vert \leq
$$
$$
\sup_{w,w'\in \W}\P(\mathcal{N}^N_C(w,w') < \mathcal{N}_{\dagger}^N(w,w')) \to 0, \text{ as } N \to +\infty,
$$
and with this we concluded the proof.
\end{proof}

\section{Proof of Theorem \ref{massconcentration}} \label{sec:teo1}

To prove Theorem \ref{massconcentration} we need to introduce the auxiliary process $(\tilde{U}_t^{N,u})_{t\in [0,+\infty)}$ that evolves as a Markov jump process taking values in the set
$
\auxst=\st\setminus\{\nlist\}
$
with initial list $u \in \auxst$ and with infinitesimal generator $\tilde{\mathcal{G}}$ defined as follows
$$
\tilde{\mathcal{G}}f(u)= \sum_{b\in \N}
e^{  u(b)}\mathbf{1}\{u(b)>0\}\left[f(\pi^{b,*}(u))-f(u)\right]+
$$
$$
\sum_{b\in \N}\mathbf{1}\{\pi^{b,\dagger}(u) \neq \nlist \}\left[f(\pi^{b,\dagger}(u))-f(u)\right],
$$
for any  bounded function $f:\auxst \to \mathbb{R}$.

\begin{remark} \label{remark1}
In general, the processes $(\tilde{U}_t^{N,u})_{t\in [0,+\infty)}$  and $(U_t^{N,u})_{t\in [0,+\infty)}$ have the same jump rates. The only exception is that $(\tilde{U}_t^{N,u})_{t\in [0,+\infty)}$ can not jump from a list in which only one neuron has non-null membrane potential to the null list.

This implies that the processes $(U_t^{N,u})_{t\in [0,+\infty)}$ and $(\tilde{U}_t^{N,u})_{t\in [0,+\infty)}$ can be coupled in such way that 
$$
\tilde{U}_s^{N,u}=U_s^{N,u} \text{, for all } s \in \big[0,\tau^{N,u}
\big).
$$
\end{remark}

For the auxiliary process $(\tilde{U}_t^{N,u})_{t\in [0,+\infty)}$, let us introduce some extra notation.

\vspace{0.2cm}

\noindent \textbf{Extra notation - auxiliary process}
\begin{itemize}
\item
Denote $\tilde{T}_0=0$ and for $n =1,2,\ldots$, denote $\tilde{T}_n$ the successive jumping times of the process $(\tilde{U}_t^{N,u})_{t\in [0,+\infty)}$, namely 
$$
\tilde{T}_n= \inf\left\{t>\tilde{T}_{n-1}:\tilde{U}_t^{N,u} \neq \tilde{U}^{N,u}_{\tilde{T}_{n-1}}\right\}.
$$

\item For $n =1,2,\ldots$, we define $\tilde{A}_n \in \N$ and $\tilde{O}_n \in \{*,\dagger\}$ as the pair such that
$$
\tilde{U}_{\tilde{T}_n}^{N,u}=\pi^{\tilde{A}_n,\tilde{O}_n}\left(\tilde{U}^{N,u}_{\tilde{T}_{n-1}}\right).
$$
The pair $(\tilde{A}_n, \tilde{O}_n)$ is exactly the index of the map used to transform the list of membrane potentials at time $\tilde{T}_n$ describing the change in the auxiliary process at this time.


\item The leakage times are defined as $\tilde{T}_0^{\dagger}=0$ and for $n\geq 1$,
    $$
    \tilde{T}_n^{\dagger}=\inf\{\tilde{T}_{m}>\tilde{T}_{n-1}^{\dagger}:\tilde{O}_{m}=\dagger\}.
    $$
    
\item The spiking times are defined as $\tilde{T}_0^{*}=0$ and for $n\geq 1$,
    $$
    \tilde{T}_n^{*}=\inf\{\tilde{T}_{m}>\tilde{T}_{n-1}^{*}:\tilde{O}_{m}=*\}.
    $$
    
\item For any time interval $I \subset [0,+\infty)$, the counting measures indicating the number of leakage
times and spiking times that occurred during the time interval $I$ are defined as
$$
\tilde{Z}^{\dagger}(I)= \sum_{j=1}^{+\infty}\mathbf{1}\{\tilde{T}_j^{\dagger} \in I\} \ \ \ \ \text{ and } \ \ \ \ 
\tilde{Z}^{*}(I)= \sum_{j=1}^{+\infty}\mathbf{1}\{\tilde{T}_j^{*} \in I\}.
$$
\item In the next proposition, we prove that $(\tilde{U}_t^{N,u})_{t\in [0,+\infty)}$ has an unique invariant probability measure. We use the symbol $\mu^{N}$ to denote this probability measure.
\end{itemize}

\begin{proposition} \label{ergodic}
The process $(\tilde{U}_t^{N,u})_{t\in [0,+\infty)}$ is ergodic. 
\end{proposition}
\begin{proof}
Let $l \in \ladder$ satisfies $l(a)=a-1$, for all $a \in \N$.
For any $u \in \auxst$, we have that
$$
l=\pi^{1,*} \circ \pi^{2,*} \circ \ldots \circ \pi^{N,*}(u),
$$
and then, if the event $\bigcap_{j=1}^{N}\{\tilde{A}_{j}=N-j+1, \tilde{O}_j=*\}$ occurs, then
$
\tilde{U}^{N,u}_N=l.
$

Let
$$
\tilde{\mathcal{N}}^{N,u}=\inf\{n \geq 1:\tilde{U}_{\tilde{T}_n}^{N,u}=l\}.
$$
As in Theorem \ref{teo: extinction}, for any $u' \in \auxst$, we have that
\begin{align*}
    &\P\Big(
    \tilde{\mathcal{N}}^{N,u}\leq n+2N-1\ \Big\vert\ \tilde{U}^{N,u}_{\tilde{T}_n}=u'\Big) \geq \\ 
    &\P\left(\tilde{U}_{\tilde{T}_{n+N-1}}^{N,u} \in \ladder \Big\vert \ \tilde{U}_{\tilde{T}_n}^{N,u}=u'\right)\P\left(\tilde{\mathcal{N}}^{N,u}\leq n+2N-1\ \Big\vert\ \tilde{U}_{\tilde{T}_{n+N-1}}^{N,u} \in \ladder\right).
\end{align*}
Moreover, the right-hand side of the equation above is bounded above by $[2(N-1)]^{-(N-1)}\tilde{\epsilon}$, where
$$
\tilde{\epsilon}=\min\left\{\P\left(\bigcap_{j=1}^{N}\{\tilde{A}_{n+j}=N-j+1, \tilde{O}_{n+j}=*\} \ \Big\vert\ \tilde{U}_{T_n}^{N,u}=v\right) : v \in \ladder \right\}.
$$

We conclude that for any $n\geq 1$,
$$
\P(\tilde{\mathcal{N}}^{N,l}\geq n) \leq \P\big((2N -1)\times \text{Geom}\big([2(N-1)]^{-(N-1)}\tilde{\epsilon}\big) \geq n\big),
$$
where $\text{Geom}(r)$ denotes a random variable with geometric distribution assuming values in $\{1,2,\ldots\}$ and with mean $1/r$. This implies that $\mathbb{E}(\tilde{\mathcal{N}}^{N,l})< +\infty$ and then,
$(\tilde{U}^{N,u}_{\tilde{T}_n})_{n\geq 0}$ is a positive-recurrent Markov chain. The jump rate of the process $(\tilde{U}_t^{N,u})_{t\in [0,+\infty)}$ satisfies
$$
\sum_{a \in \N}\mathbf{1}\{u'(a)>0\}\left(e^{ u'(a)}+\mathbf{1}\{\pi^{a,\dagger}(u') \neq \nlist\}\right) \geq e,
$$
for any $u' \in \auxst$. 
Putting all this together we conclude that $(\tilde{U}_t^{N,u})_{t\in [0,+\infty)}$ is ergodic.
\end{proof}

The proof of Theorem \ref{massconcentration} is based on two lemmas. In Lemma \ref{massconcentration1} we prove that the invariant measure of the auxiliary process gets concentrated in the set $\W$ as $N\to+\infty$. We show that there are events that occur with probability $1$ as $N\to +\infty$ that lead the process from any initial list to $\W$ by sequentially reaching sets that are approaching $\W$. As a consequence, in Corollary \ref{coro2} we obtain a bound for the first time in which the auxiliary process reaches $\ladder$ and in Corollary \ref{coro3} we show that the process $(U_t^{N,w})_{t\in [0,+\infty)}$ reaches $\ladder$ instantaneously as $N\to +\infty$ for any $w \in \W$. In Lemma \ref{propcoupling} we prove that starting from a ladder list, the auxiliary process is in $\W$ with probability $1$ as $N\to +\infty$. Putting all this together we are able to prove Theorem \ref{massconcentration}.

\begin{lemma} \label{massconcentration1}
The invariant probability measure $\mu^{N}$ of the auxiliary process satisfies
$$
\mu^{N}(\W) \to 1, \text{ as } N \to +\infty.
$$
\end{lemma}
\begin{proof} To prove Lemma \ref{massconcentration1}, we will first show that there exists sets $S_N^{(1)}, S_N^{(2)}$ and  $S_N^{(3)}$ such that
$$
S_N^{(1)} \supset S_N^{(2)} \supset S_N^{(3)} \supset \W
$$
and for any $j\in\{1,2,3\}$,
$$
\mu^{N}(S_N^{(j)}) \to 1, \text{ as } N \to +\infty.
$$

Let
$$
S_N^{(1)}=\{u \in \auxst: \vert\{a \in \N: u(a)=0\}\vert \leq N^{\frac{1}{2}}\}
$$
and consider the following events
$$
E^{(1)}_{N,1}=\{\tilde{Z}^{*}([0,N^{\frac{1}{2}}])\geq 1\},
$$
$$
E^{(1)}_{N,2}=\{\tilde{Z}^{\dagger}([0,N^{\frac{1}{2}}])\leq N^2\},
$$
$$ 
E^{(1)}_{N,3}=\bigcap_{j=1}^{N^2/\lfloor\frac{N^{1/2}}{2}\rfloor}\left\{\tilde{Z}^{*}\left(\left[\tilde{T}^{\dagger}_{(j-1)\lfloor\frac{N^{1/2}}{2}\rfloor +1},\tilde{T}^{\dagger}_{j\lfloor\frac{N^{1/2}}{2}\rfloor }\right]\right) \geq 1\right\}.
$$

For any $u \in \auxst$, the rate in which the process has a leakage is bounded above by $N$. Moreover, the rate in which the process has a spike is bounded bellow by $e$. This implies that
\begin{equation} \label{eqmc1}
\P(E^{(1)}_{N,1}) \geq \P\left(\xi^{\{e\}}\leq N^{\frac{1}{2}} \right) \to 1, \text{ as } N \to +\infty,
\end{equation}
\begin{equation} \label{eqmc2} 
\P(E^{(1)}_{N,2}) \geq \P\left(\sum_{j=1}^{N^2}\xi_j^{\{N\}}\geq N^{\frac{1}{2}} \right) \to 1, \text{ as } N \to +\infty.
\end{equation}
For any list $u\in \auxst$ and for any instant $n\geq 1$, we have that
$$
\P\left(\tilde{O}_n=* \big\vert \tilde{U}_{\tilde{T}_{n-1}}^{N,u'}=u\right) \geq \frac{1}{2}.
$$
This implies that for any  initial list $u \in \auxst$ and for any $j=1,\ldots, N^2/\lfloor\frac{N^{1/2}}{2}\rfloor$,
$$
\P\left(\tilde{Z}^{*}\left(\left[\tilde{T}^{\dagger}_{(j-1)\lfloor\frac{N^{1/2}}{2}\rfloor +1 },\tilde{T}^{\dagger}_{j\lfloor\frac{N^{1/2}}{2}\rfloor }\right]\right) \geq 1\right) \geq \P\left(\text{Geom}\left(\frac{1}{2}\right) \leq \Big\lfloor\frac{ N^{1/2}}{2} \Big\rfloor-1\right),
$$
where $\text{Geom}\left(\frac{1}{2}\right)$ is a random variable with geometric distribution assuming values in $\{1,2,\ldots\}$ and with mean $2$. Therefore,
\begin{equation} \label{eqmc3}
\P(E^{(1)}_{N,3}) \geq \left(1-2^{-\lfloor\frac{ N^{1/2}}{2} \rfloor+1}\right)^{N^2/\lfloor\frac{N^{1/2}}{2}\rfloor} 
\to 1, \text{ as } N \to +\infty.
\end{equation}

If the event $E^{(1)}_{N,1} \cap E^{(1)}_{N,2} \cap E^{(1)}_{N,3}$ occurs, then until time $N^{1/2}$ the process has at least one spiking time and at most $\lfloor N^{1/2} \rfloor-1$ successive leakage times (with no spiking times in between). This implies that $\tilde{U}^{N,u}_{N^{1/2}} \in S_N^{(1)}$.
Since the inequalities of Equations \eqref{eqmc1} , \eqref{eqmc2} and  \eqref{eqmc3} holds for any $u \in \auxst$ and they do not depend on $u$, it follows that
$$
\sup_{u \in \auxst}\P(\tilde{U}^{N,u}_{N^{1/2}} \notin S_N^{(1)}) \to 0, \text{ as } N \to +\infty,
$$
and as a consequence,
$$
\mu^N(S_N^{(1)})=\sum_{u \in \auxst}\mu^N(u)\P(\tilde{U}^{N,u}_{N^{1/2}} \in S_N^{(1)}) \to 1, \text{ as } N \to +\infty.
$$

Let
$$
S_N^{(2)}= \left\{u \in \auxst: \exists \  a_1, \ldots, a_{\lceil N^{1/2} \rceil} \in \N \text{ s.t. } 1 \leq u(a_1) < \ldots < u\left(a_{ \lceil N^{1/2}\rceil}  \right)\right\}
$$
and consider the following events
$$ 
E^{(2)}_{N,1}=\bigcap_{j=1}^{N^2/\lfloor\frac{N^{1/2}}{2}\rfloor}\left\{\tilde{Z}^{*}\left(\left[\tilde{T}^{\dagger}_{(j-1)\lfloor\frac{N^{1/2}}{2}\rfloor +1},\tilde{T}^{\dagger}_{j\lfloor\frac{N^{1/2}}{2}\rfloor }\right]\right) \geq 1\right\},
$$
$$
E^{(2)}_{N,2}=\{\tilde{Z}^{\dagger}([0,N^{-\frac{1}{4}}])\leq N^2\},
$$
$$
E^{(2)}_{N,3}=\{\tilde{Z}^{*}([0,N^{-\frac{1}{4}}])\geq \lceil N^{1/2}\rceil\}.
$$

As in Equation \eqref{eqmc2} and \eqref{eqmc3},
\begin{equation}  \label{eqmc21}
\P(E^{(2)}_{N,1}) \geq \left(1-2^{-\lfloor\frac{ N^{1/2}}{2} \rfloor+1}\right)^{N^2/\lfloor\frac{N^{1/2}}{2}\rfloor} 
\to 1, \text{ as } N \to +\infty,
\end{equation}
\begin{equation} \label{eqmc22} 
\P(E^{(2)}_{N,2}) \geq \P\left(\sum_{j=1}^{N^2 }\xi_j^{\{N\}}\geq N^{-\frac{1}{4}} \right) \to 1, \text{ as } N \to +\infty.
\end{equation}

For any initial list $u \in  S_N^{(1)}$, the occurrence of $E^{(2)}_{N,1}\cap E^{(2)}_{N,2}$ implies that until time $N^{-1/4}$ the rate in which the process has a spike is bounded bellow by $e(N-2\lceil N^{1/2} \rceil)$. This implies that
$$ 
\P(E^{(2)}_{N,3}) \geq \P\left(\sum_{j=1}^{\lceil N^{1/2}\rceil}\xi_j^{\{N-2\lceil N^{1/2} \rceil\}}\leq N^{-\frac{1}{4}} \right)\P\left(E^{(2)}_{N,1}\cap E^{(2)}_{N,2}\right) \to 1, \text{ as } N \to +\infty.
$$

For any initial list $u \in S_N^{(1)}$, if the event $E^{(2)}_{N,1} \cap E^{(2)}_{N,2} \cap E^{(2)}_{N,3}$ occurs, then until time $N^{-1/4}$ the process has at least $\lceil N^{1/2} \rceil $ spiking times and at most $\lfloor N^{1/2} \rfloor-1$ successive leakage times (with no spiking times in between). This implies that $\tilde{U}^{N,u}_{N^{-1/4}} \in S_N^{(2)}$.
Therefore,
$$
\sup_{u \in S_N^{(1)}}\P(\tilde{U}^{N,u}_{N^{-1/4}} \notin S_N^{(2)}) \to 0, \text{ as } N \to +\infty,
$$
and as a consequence,
$$
\mu^N(S_N^{(2)})=\sum_{u \in \auxst}\mu^N(u)\P(\tilde{U}^{N,u}_{N^{-1/4}} \in S_N^{(2)}) \to 1, \text{ as } N \to +\infty.
$$

Let
$$
S_N^{(3)}=\{u \in \auxst: u(a_j^u) \geq j-1, \text{ for all } j=1,\ldots, N\}
$$
and consider the following events
$$
E^{(3)}_{N,1}=\{\tilde{Z}^{\dagger}([0,N^{-2}])=0\},
$$
$$
E^{(3)}_{N,2}=\{\tilde{Z}^{*}([0,N^{-2}])\geq N\}.
$$
For any $u \in S_N^{(2)} $,
$$
\P(E^{(3)}_{N,1}) \geq \P\left(\xi^{\{N\}}\geq N^{-2} \right) \to 1, \text{ as } N \to +\infty,
$$
For any initial list $u \in  S_N^{(2)}$, the occurrence of the event $E^{(3)}_{N,1}$ implies that until time $N^{-2}$ the rate in which the process has a spike is bounded bellow by $e^{ \lfloor N^{1/2} \rfloor}$. This implies that
$$ 
\P(E^{(3)}_{N,2}) \geq \P\left(\sum_{j=1}^{N}\xi_j^{\{e^{ \lfloor N^{1/2} \rfloor}\}}\leq N^{-2} \right)\P\left(E^{(3)}_{N,1}\right) \to 1, \text{ as } N \to +\infty.
$$

For any initial list $u \in S_N^{(2)}$, if the event $E^{(3)}_{N,1} \cap E^{(3)}_{N,2}$ occurs, then until time $N^{-2}$ the process has at least $N$ spiking times and does not have any leakage. This implies that $\tilde{U}^{N,u}_{N^{-2}} \in S_N^{(3)}$.
Therefore,
$$
\sup_{u \in S_N^{(2)}}\P(\tilde{U}^{N,u}_{N^{-2}} \notin S_N^{(3)}) \to 0, \text{ as } N \to +\infty,
$$
and as a consequence,
$$
\mu^N(S_N^{(3)})=\sum_{u \in \auxst}\mu^N(u)\P(\tilde{U}^{N,u}_{N^{-2}} \in S_N^{(3)}) \to 1, \text{ as } N \to +\infty.
$$

Recall that
$$
\W=
\left\{u \in \st: I_N\subset \{u(a): a \in \N\}
,\ \bigcap_{a \in \N} \bigcap_{b \neq a} \{u(a) \neq u(b)\}
\right\},
$$
where $I_N=\left\{1,\ldots, N - \lfloor N^{1/2} \rfloor \right\}$,
and consider the following events
$$
E_{N,1}^{(4)}=\{\tilde{Z}^{\dagger}([0,e^{- (N- N^{1/4})}])=0\},
$$
$$
E_{N,2}^{(4)}=\left\{\sum_{j=1}^{+\infty}\mathbf{1}\{\tilde{T}_j \leq e^{- (N- N^{1/4})}, \tilde{U}^{N,u}_{\tilde{T}_{j-1}}(\tilde{A}_j) \leq N-\lfloor N^{1/2} \rfloor, O_j =* \} = 0 \right\},
$$
$$
E_{N,3}^{(4)}=\{\tilde{Z}^{*}([0,e^{- (N- N^{1/4}) }])\geq N+\lceil N^{1/2} \rceil \},
$$
$$
E_{N,4}^{(4)}=\bigcap_{j=1}^{N+ \lceil N^{1/2} \rceil }\left\{ 
\tilde{U}^{N,u}_{\tilde{T}_{j-1}}(\tilde{A}_j)\geq \max\left\{\tilde{U}^{N,u}_{\tilde{T}_{j-1}}(a): a \in \N\right\}- \lfloor N^{1/2} \rfloor \right\}.
$$
For any $u \in \auxst$,
$$
\P(E^{(4)}_{N,1}) \geq \P\left(\xi^{\{N\}}\geq e^{- (N- N^{1/4})} \right) \to 1, \text{ as } N \to +\infty.
$$
The rate in which the process has a spike of a neuron that in the moment of the spike have membrane potential smaller or equal $N-\lfloor N^{1/2} \rfloor$ is bounded above by $N e^{  N- \lfloor N^{1/2} \rfloor}$. Therefore,
$$ 
\P(E^{(4)}_{N,2}) \geq \P\left(\xi^{\{N e^{ N- \lfloor N^{1/2} \rfloor}\}} >e^{- (N- N^{1/4})} \right) \to 1, \text{ as } N \to +\infty.
$$

For any $u \in S_N^{(3)}$, the occurrence of the event $E^{(4)}_{N,1}$ implies that the rate in which the process has a spike is bounded bellow by $e^{ (N-1)}$. Therefore,
$$ 
\P(E^{(4)}_{N,3}) \geq \P\left(\sum_{j=1}^{ N+\lceil N^{1/2} \rceil }\xi_j^{\{e^{ (N-1)}\}}\leq e^{- (N- N^{1/4})} \right)\P\left(E^{(4)}_{N,1}\right) \to 1, \text{ as } N \to +\infty.
$$
Moreover, the probability
\begin{equation}\label{e4n2bound}
\P\left(\tilde{U}^{N,u}_{\tilde{T}_{0}}(\tilde{A}_1)\geq \max\left\{u(a): a \in \N\right\}- \lfloor N^{1/2} \rfloor  \right)
\end{equation}
is minimized when the difference between the membrane potential of the neuron with greatest potential and the membrane potential of the other neurons is $\lfloor N^{1/2} \rfloor+1$. This implies that \eqref{e4n2bound} is bounded bellow by
$$
\frac{e^{ \lfloor N^{1/2} \rfloor}}{e^{ \lfloor N^{1/2} \rfloor}+2(N-1)},
$$
and therefore, 
$$ 
\P(E^{(4)}_{N,4}) \geq \left(\frac{e^{ \lfloor N^{1/2} \rfloor}}{e^{ \lfloor N^{1/2} \rfloor}+2(N-1)}\right)^{ N+\lceil N^{1/2} \rceil}\P\left(E^{(4)}_{N,1}\cap E^{(4)}_{N,3}\right) \to 1, \text{ as } N \to +\infty.
$$

For any initial list $u \in  S_N^{(3)}$, if the event $E^{(4)}_{N,1} \cap \ldots \cap E^{(4)}_{N,4}$ occurs, then until time $e^{- (N-N^{1/4})}$ the process has at least $N+\lceil N^{1/2} \rceil$ spiking times, does not have any leakage of membrane potential and does not have any spike of a neuron with membrane potential smaller or equal $ N- \lfloor N^{1/2} \rfloor $. This implies that 
$$
\left\{1,\ldots, N - \lfloor N^{1/2} \rfloor \right\} \subset \left\{\tilde{U}^{N,u}_{e^{- (N-N^{1/4})}}(a): a \in \N\right\}.
$$
Moreover, the occurrence of the events $E^{(4)}_{N,1}\cap E^{(4)}_{N,3} \cap E^{(4)}_{N,4}$ implies that  all neurons spikes at least once in the first $ N+\lceil N^{1/2} \rceil$ steps of the process.  This implies that $\tilde{U}^{N,u}_{e^{- (N-N^{1/4})}}(a) \neq \tilde{U}^{N,u}_{e^{- (N-N^{1/4})}}(a')$, for all $a\neq a'$. We conclude that  $\tilde{U}^{N,u}_{e^{- (N-N^{1/4})}} \in \W$.

Therefore,
$$
\sup_{u \in S_N^{(3)}}\P\left(\tilde{U}^{N,u}_{e^{- (N-N^{1/4})}} \notin \W\right) \to 0, \text{ as } N \to +\infty,
$$
and as a consequence,
$$
\mu^N(\W)=\sum_{u \in \auxst}\mu^N(u)\P\left(\tilde{U}^{N,u}_{e^{- (N-N^{1/4})}} \in \W\right) \to 1, \text{ as } N \to +\infty.
$$
\end{proof}

From Lemma \ref{massconcentration1} it follows Corollaries \ref{coro2} and \ref{coro3}, that are used to prove Theorems \ref{massconcentration} and \ref{metastable}.

\begin{corollary} \label{coro2} The following holds
$$
\inf_{u \in \auxst}\P\left(\inf\left\{\tilde{T}_n: \tilde{U}_{\tilde{T}_n}^{N,u} \in \ladder\right\} \leq t_N+e^{-(N-N^{1/2})}\right)\to 1, \text{ as } N \to +\infty,
$$
where $t_N=N^{1/2}+N^{-1/4}+N^{-2}+e^{-(N-N^{1/4})}$.
\end{corollary}
\begin{proof}
First, note that for any $u \in \auxst$, 
$$
\P\left(\inf\left\{\tilde{T}_n: \tilde{U}_{\tilde{T}_n}^{N,u} \in \W\right\} \leq t_N
\right) \geq 
$$
$$
\P\left(\tilde{U}_{N^{1/2}}^{N,u} \in S_N^{(1)},\ \tilde{U}_{N^{1/2}+N^{-1/4}}^{N,u} \in S_N^{(2)}, \ \tilde{U}_{N^{1/2}+N^{-1/4}+N^{-2}}^{N,u} \in S_N^{(3)}, \ \tilde{U}_{t_N}^{N,u} \in \W\right). 
$$
Remark \ref{remarkladderreach} implies that for any $w \in \W$,
$$
\inf_{w \in \W}\P\left(\inf\left\{\tilde{T}_n: \tilde{U}_{\tilde{T}_n}^{N,w} \in \ladder\right\} \leq e^{-(N-N^{1/2})} \right) \to 1, \text{ as } N \to +\infty.
$$
We conclude the proof by putting all this together with Lemma \ref{massconcentration1} and Markov property.
\end{proof}

\begin{corollary} \label{coro3} 
The following holds
$$
\inf_{u \in S_N^{(0)}}\P\left(\inf\left\{T_n: U_{T_n}^{N,u} \in \ladder\right\} \leq t'_N\right)\to 1, \text{ as } N \to +\infty,
$$
where $t'_N=N^{-1/4}+N^{-2}+e^{-(N-N^{1/4})}+e^{-(N-N^{1/2})}$.
\end{corollary}
\begin{proof}
Note that starting from any list $u \in S_N^{(0)}$, as in the proof of Lemma \ref{massconcentration1} we have that
$$
\P\left(E_{N,1}^{(2)} \cap E_{N,2}^{(2)} \cap E_{N,3}^{(2)}\right)\to 1, \text{ as } N \to +\infty.
$$
Then, as in Corollary \ref{coro2} we have that
\begin{equation} \label{coro3eq}
\inf_{u \in S_N^{(0)}}\P\left(\inf\left\{\tilde{T}_n: \tilde{U}_{\tilde{T}_n}^{N,u} \in \ladder\right\} \leq t'_N 
\right)\to 1, \text{ as } N \to +\infty.
\end{equation}

By the definition of the events 
$
E_{N,1}^{(2)}, E_{N,1}^{(3)} \text{ and } E_{N,1}^{(4)},
$
we have that
$$
\inf_{u \in S_N^{(0)}}\P\left(\inf\left\{\tilde{T}_n: \left\lvert a \in \N:\tilde{U}_{\tilde{T}_n}^{N,u}(a)>0\right\rvert =1\right\} > t'_N\right)\to 1, \text{ as } N \to +\infty.
$$
By Remark \ref{remark1} and the coupling construction, it follows that
$$
\inf_{u \in S_N^{(0)}}\P\left(U_t^{N,u}=\tilde{U}_t^{N,u}, \text{ for all } t \in [0,t'_N
]\right) \to 1, \text{ as } N \to +\infty.
$$
Therefore, we can replace $\tilde{U}$ by $U$ and $\tilde{T}_n$ by $T_n$ on Equation \eqref{coro3eq} and with this we concluded the proof.

\end{proof}

\begin{lemma} \label{propcoupling}
For any $N \geq 2$, for any list $l \in \ladder$ and for any $s>0$,
$$
\P(\tilde{U}^{N,l}_s \in \W^c) \leq \frac{\mu^N(\W^c)}{\mu^N(\W)}+\delta(N,s),
$$
where $\lim_{N\to +\infty}\delta(N,s)=0$, for any $s>0$.
\end{lemma}
\begin{proof} For any $s>0$,
$$
\mu^{N}(\W)=\sum_{u \in \W}\mu^{N}(u)\P(\tilde{U}^{N,u}_s \in \W)+\sum_{u \in \auxst \setminus \W}\mu^{N}(u)\P(\tilde{U}^{N,u}_s \in \W).
$$
By Remark \ref{remark1}, for any $l \in \ladder$ and $w \in \W$ we have
$$
\P(\tilde{U}^{N,w}_s \in \W)\leq
$$
$$
\P(\tilde{U}^{N,l}_s \in \W)+
\P\left(\{T_{\mathcal{N}^N_C(l,w)}(l,w)>s\} \cup \{\mathcal{N}^N_C(l,w) > \mathcal{N}^N_{\dagger}(l,w)\}\right).
$$
Considering
$$
\delta(N,s)=\sup_{l \in \ladder}\sup_{w \in \W}\P\left(\{T_{\mathcal{N}^N_C(l,w)}(l,w)>s\}\cup \{\mathcal{N}^N_C(l,w) > \mathcal{N}^N_{\dagger}(l,w)\}\right),
$$
by 
Lemma \ref{couplingtime} and Corollary \ref{coro} it follows that $\displaystyle\lim_{N\to +\infty}\delta(N,s)=0$, for any $s>0$.
Moreover,
$$
\sum_{u \in \auxst \setminus \W}\mu^{N}(u)\P(\tilde{U}^{N,u}_s \in \ladder) \leq 1-\mu^{N}(\W).
$$
This implies that
$$
\mu^{N}(\W) \leq \mu^{N}(\W)(\P(\tilde{U}^{N,l}_s \in \W)+\delta(N,s))+(1-\mu^{N}(\W)),
$$
and therefore, 
$$
\P(U^{N,l}_s \in \W)\geq \frac{\mu^{N}(\W)-(1-\mu^{N}(\W))}{\mu^{N}(\W)}-\delta(N,s).
$$
With this we concluded the proof of Lemma \ref{propcoupling}.

\end{proof}

\begin{proof} Now we will prove Theorem \ref{massconcentration}.

By Remark \ref{remark1} and the invariance by permutation of the process it follows that for any $u \in S_N^{(0)}$, for any $l \in \ladder$ and for any $t>0$,
$$
\P\left(U_{t }^{N,u}\in \st\setminus \W \ \Big\vert\ \tau^{N,u}>t \right)\leq
$$
$$
\P(\inf\{t>0:U_t^{N,u} \in \ladder\} > t/2)+\sup_{s \in [t/2,t]}\P\left(\tilde{U}_{s}^{N,l}\in \auxst\setminus\W\right).
$$
By Corollary \ref{coro3},
$$
\sup_{u \in S_N^{(0)}}\P(\inf\{t>0:U_t^{N,u} \in \ladder\} > t/2) \to 0, \text{ as } N \to +\infty.
$$
By Lemma \ref{propcoupling},
$$
\sup_{s \in [t/2,t]}\P\left(\tilde{U}_{s}^{N,l}\in \auxst\setminus\W\right) \leq \frac{\mu^N(\W^c)}{\mu^N(\W)}+\delta(N,t/2).
$$
By Lemmas \ref{massconcentration1} and \ref{propcoupling} it follows that
$$
\lim_{N\to +\infty} \delta(N,t/2) =\lim_{N\to +\infty}\mu^N(\W^c)=0
$$
and with this we concluded the proof.

\end{proof}

\begin{remark}\label{remark3}
For any $N \geq 2$, $\P(\inf\{t>0:U_t^{N,u} \in \ladder\} > t/2)$ and $\delta(N,t/2)$ decreases with $t$. This implies that 
for any $(t_N:N \geq 2)$ such that $\displaystyle\lim_{N\to +\infty}t_N=+\infty$, we have
$$
\inf_{u \in S_N^{(0)}}\P\left(U_{t_N }^{N,u}\in \W \ \vert \  \tau^{N,u}>t_N \right) \to 1, \text{ as } N \to +\infty.
$$

\end{remark}

\begin{remark}
The important feature used in the proof of Theorem \ref{massconcentration} is that the exponential growth of the spiking rate produces the mechanism that leads the process to the set $\W$.
This mechanism also allows the process to instantaneously reach the set of ladder lists when starting from $\W$ as $N\to +\infty$, which is important to prove Theorem \ref{metastable}. The choice of the base is not important. The only requirement of the base is to be greater than $1$.
\end{remark}

\section{Proof of Theorem \ref{metastable}}
\label{sec:teo2}

For any fixed $l \in \ladder$, let $c_{N,l}$ be the positive real number such that
\begin{equation}\label{cbeta}
\P(
\tau^{N,l}>c_{N,l})=e^{-1}.
\end{equation}
Due to the invariance by permutation of the process, it is clear that $c_{N,l}=c_{N,l'}$, for any pair of lists $l$ and $l'$ belonging to $\ladder$. Therefore, in what follows we will omit to indicate $l$ in the notation of $c_{N}.$

To prove Theorem \ref{metastable}, we prove in Proposition \ref{explim2} that for any sequence $(l_N \in \ladder: N\geq 2)$,
$$
\frac{\tau^{N,l_N}}{c_{N}}\to \text{Exp}(1), \text{ as } N \to +\infty,
$$
where $\text{Exp}(1)$ is a random variable exponentially distributed with mean $1$. We will prove that the limiting distribution satisfies the memoryless property, which characterizes the exponential distribution. For this, it is necessary to prove that $c_N \to +\infty$ as $N\to +\infty$. This is the content of Proposition \ref{cnbound}. Lemmas \ref{lemmasup} and \ref{prop7} give the necessary conditions to replace $c_N$ by $\E[\tau^{N,l_N}]$ in Proposition \ref{explim2}. Using the fact that the process starting from any list in $\W$ will quickly reach $\ladder$ as $N\to +\infty$ we finish the proof of Theorem \ref{metastable}.


\begin{proposition} \label{cnbound}
For any $N \geq 3$,
$$
c_{N} \geq \frac{N-1+e^{ (N-2)}}{(N-1)^3}.
$$
\end{proposition}

\begin{proof}
For a initial list $l \in \ladder$, let
$$
\tau_-^N=\inf\left\{T_n: O_n=\dagger,\bigcup_{j=1}^{N-1} \{O_{n-j}=\dagger\}\right\}.
$$ 
We have
$$
\tau_-^N=\sum_{j=1}^{G}(T^{\dagger}_j-T^{\dagger}_{j-1}),
$$
where 
$
G=\inf\{j: Z^{*}([T^{\dagger}_{j-1},T^{\dagger}_j])\leq N-2\}.
$

The rate in which the process has a leakage is bounded above by $N-1$. Therefore, for any $j\geq 1$ and for any $s>0$,
$$
\P(T^{\dagger}_j-T^{\dagger}_{j-1}>s) \geq \P(\xi^{\{N-1\}}>s).
$$

Recall that
$$
S_N^{(3)}=\{u \in \st: u(a_j^u) \geq j-1, \text{ for any } j=1,\ldots, N\}.
$$
For any initial list $w \in \W$,
we have that
$$
U_t^{N,w} \in S_N^{(3)}, \text{ for any } t< T_1^{\dagger}.
$$
Moreover, for any initial list $u \in \st \setminus \{\nlist\}$, if $O_1=\ldots = O_{N-1}=*$, then $U_{T_{N-1}}^{N,u} \in S_N^{(3)}$.
Together with Markov property, this implies that for any $m\geq 1$ and for any $j\geq 1$,
$$
\P(Z^{*}([T_{j-1}^{\dagger},T^{\dagger}_j])\leq N-2 \ \vert \ T_{j-1}^{\dagger}=T_m, G\geq j )=
$$
$$
\P\left(\bigcup_{j=1}^{N-1}\{O_{j+m}=\dagger\} \ \Big\vert \ O_m=\dagger, U_{T_{m-1}}^{N,u} \in  S_N^{(3)} \right).
$$
The probability on the right-hand side of equation above is bounded above by
$$
\lambda_N=(N-1) \times \frac{N-1}{N-1+e^{ (N-2)}}.
$$
Therefore, for any $s>0$,
$$
\P(\tau_-^{N}>s ) \geq \P\left(\sum_{j=1}^{\text{Geom}(\lambda_N)}\xi_j^{\{N-1\}}> t\right),
$$
where $\text{Geom}(\lambda_N)$ is a random variable independent of $(\xi_j^{\{N-1\}})_{j\geq 1}$ with Geometric distribution assuming values in $\{1,2,...\}$ with mean $1/\lambda_N$.
This implies that
$$
\P(\tau_-^N>s) \geq \P\left(\xi^{\{\lambda_{N}(N-1)\}}>s\right).
$$

Therefore,
\begin{align*}
e^{-1}=\P(\tau^{N,l}>c_{N}) \geq \P(\tau_-^N>c_{N}) \geq e^{-c_N \lambda_{N}(N-1)} ,
\end{align*}
and then,
$$
c_{N}\geq \frac{1}{\lambda_{N}(N-1)}.
$$
\end{proof}

To prove Theorem \ref{metastable}, we prove Proposition \ref{explim2} which is interesting by itself.

\begin{proposition}\label{explim2} For any sequence $(l_N \in \ladder: N\geq 2)$,
$$
\frac{\tau^{N,l_N}}{c_{N}}\to \text{Exp}(1), \text{ as } N \to +\infty,
$$
where $\text{Exp}(1)$ is a random variable exponentially distributed with mean $1$.
\end{proposition}

\begin{proof}
First of all, we will prove that 
for any sequence $(l_N \in \ladder: N\geq 2)$ and
for any pair of positive real numbers $s,t\geq 0$, the following holds
\begin{equation}  \label{explim1}
 \lim_{N \to +\infty}\left\lvert\P\left(\frac{\tau^{N,l_N}}{c_{N}}>s+t \right)-\P\left(\frac{\tau^{N,l_N}}{c_{N}}>s \right)\P\left(\frac{\tau^{N,l_N}}{c_{N}}>t \right) \right\rvert=0.   
\end{equation}


Indeed, for any $N\geq 2$ and for any $l \in \ladder$,
$$
\left\lvert\P\left(\frac{\tau^{N,l}}{c_{N}}>s+t \right)-\P\left(\frac{\tau^{N,l}}{c_{N}}>s \right)\P\left(\frac{\tau^{N,l}}{c_{N}}>t \right) \right\rvert\leq
$$
\begin{equation} \label{eqmeta1}
\sum_{u \in \st\setminus\{\nlist\}}\P\left(U_{c_{N}s }^{N,l}=u,\frac{\tau^{N,l}}{c_{N}}>s\right)\left\lvert\P\left(\frac{\tau^{N,u}}{c_{N}}>t \right)-\P\left(\frac{\tau^{N,l}}{c_{N}}>t \right)\right\rvert. 
\end{equation}
The right-hand side of Equation \eqref{eqmeta1} is equal
$$
\sum_{u \in \W}\P\left(U_{c_{N}s }^{N,l}=u,\frac{\tau^{N,l}}{c_{N}}>s\right)\left\lvert\P\left(\frac{\tau^{N,u}}{c_{N}}>t \right)-\P\left(\frac{\tau^{N,l}}{c_{N}}>t \right)\right\rvert+
$$
$$
\sum_{u \in \st\setminus\{\W \ \cup \ \nlist\}}\P\left(U_{c_{N}s }^{N,l}=u,\frac{\tau^{N,l}}{c_{N}}>s\right)\left\lvert\P\left(\frac{\tau^{N,u}}{c_{N}}>t \right)-\P\left(\frac{\tau^{N,l}}{c_{N}}>t \right)\right\rvert \leq
$$
\begin{equation} \label{eqpropexp2}
\sup_{w \in \W}\left\lvert\P(\tau^{N,l}>c_N t)-\P(\tau^{N,w}>c_N t) \right\rvert+ \P\left(U_{c_{N}s }^{N,l}\in \st\setminus \W,\tau^{N,l}>c_{N}s \right).
\end{equation}
By Theorem \ref{massconcentration}, Remark \ref{remark3} and Propositions \ref{unifconvergence} and \ref{cnbound}, Equation \eqref{eqpropexp2} and the invariance by permutation of the process implies \eqref{explim1}.

By definition, for any $N\geq 2$ and for any $l \in \ladder$,
$$
\P\left(\frac{\tau^{N,l}}{c_{N}} > 1\right)=e^{-1}.
$$
Iterating \eqref{explim1} with $t=s=2^{-n}$, for $n=1,2,\ldots$, we have that for any sequence $(l_N \in \ladder: N\geq 2)$,
\begin{equation*} 
\P\left(\frac{\tau^{N,l_N}}{c_{N}} > 2^{-n}\right) \to e^{-2^{-n}}, \text{ as } N \to +\infty.
\end{equation*}
More generally, we have that for any
$$
t \in \left\{\sum_{n=1}^{m}b(n)2^{-n}: b(n) \in \{0,1\}, n=1,...,m, m \geq 1\right\}
$$
is valid that
\begin{equation}
\P\left(\frac{\tau^{N,l_N}}{c_{N}} > t\right) \to e^{-t}, \text{ as } N \to +\infty.
\label{convergence}
\end{equation}
Any real number $r \in (0,1)$ has a binary representation
$$
r= \sum_{n=1}^{+\infty}b(n)2^{-n},
$$
where for any $n\geq 1$, $b(n) \in \{0,1\}$. Therefore, 
the monotonicity of 
$$
t \to \P\left(\frac{\tau^{N,l_N}}{c_{N}} > t\right)
$$
implies that the convergence in \eqref{convergence} is valid for any $t \in (0,1)$. Moreover, for any positive integer $n\geq 1$,
Equation \eqref{explim1} implies that
\begin{equation*}
    \P\left(\frac{\tau^{N,l_N}}{c_{N}} > n\right) \to e^{-n}, \text{ as } N \to +\infty.
\end{equation*}
We conclude that \eqref{convergence} is valid for any $t >0$.
\end{proof}

\begin{remark} \label{remark2} For any $N \geq 2$ and for any $l_N \in \ladder$, the function $f_{N}:[0,+\infty) \to [0,1]$ given by
$$
f_{N}(t)=\P\left(\frac{\tau^{N,l_N}}{c_{N}} > t \right)
$$
is monotonic. Also, by Proposition \ref{explim2}, it converges pointwise as $N \to +\infty$ to a continuous function. Therefore, for any $(\epsilon_N:N \geq 2)$ such that $\displaystyle\lim_{N\to +\infty}\epsilon_N=0$ , for any $t>0$ and for any sequence $(l_N \in \ladder: N\geq 2)$, we have
$$
\lim_{N \to +\infty}\P\left(\frac{\tau^{N,l_N}}{c_{N}} >t+\epsilon_{N} \right)=\lim_{N \to +\infty}\P\left(\frac{\tau^{N,l_N}}{c_{N}} >t-\epsilon_{N} \right)=e^{-t}.
$$
\end{remark}

To prove Theorem \ref{metastable}, we need the two following lemmas.

\begin{lemma} \label{lemmasup}
For any $t>0$,
$$
\lim_{N\to +\infty}\sup_{u \in \st \setminus \{\nlist\}}\P\left(\frac{\tau^{N,u}}{c_{N}} > t\right)\leq e^{-t}.
$$
\end{lemma}
\begin{proof}
For any $u \in \st \setminus \ladder$ and for any $N \geq 2$, consider the event
$$
E_{N,u}=\big\{\min\{\tau^{N,u}, \inf\{T_n: U_{T_n}^{N,u} \in \ladder\}\} \leq t'_N\big\},
$$
where $t'_N=N^{1/2}+N^{-1/4}+N^{-2}+e^{-(N-N^{1/4})}+e^{-(N-N^{1/2})}$.
We have that 
\begin{equation}\label{eqsup}
\P\left(\frac{\tau^{N,u}}{c_{N}} > t\right) = \P\left(\frac{\tau^{N,u}}{c_{N}} > t,E_{N,u}, \inf\{T_n: U_{T_n}^{N,u} \in 
\ladder\}<\tau^{N,u}\right)+
\end{equation}
$$
\P\left(\frac{\tau^{N,u}}{c_{N}} > t,E_{N,u}, \tau^{N,u}<\inf\{T_n: U_{T_n}^{N,u} \in 
\ladder\}\right)+\P\left(\frac{\tau^{N,u}}{c_{N}} > t,E_{N,u}^c\right).
$$
By Proposition \ref{cnbound}, there exists $N_t>0$ such that for any $N > N_t$, we have that $c_{N}t> t'_N$. This implies that, for any $ N> N_t$ and for any $u \in \st \setminus \ladder$,
$$
\P\left(\frac{\tau^{N,u}}{c_{N}} > t,E_{N,u}, \tau^{N,u}< \inf\{T_n: U_{T_n}^{N,u} \in \W\}\right)=0.
$$
Considering $l \in \ladder$, for any $u \in \st \setminus\{\ladder \cup \nlist\}$ and for any $N>N_t$, the left-hand side of Equation \eqref{eqsup} is bounded
above by
\begin{equation}\label{equb}
\P\left(\frac{\tau^{N,l}}{c_{N}} > t-\frac{t'_N}{c_{N}} \right)
+\P(E_{N,u}^c). 
\end{equation}
By Remark \ref{remark1} and Corollary \ref{coro2}, it follows that
$$
\lim_{N \to +\infty}\sup_{u \in \st \setminus\{\nlist\}}\P\left(E_{N,u}^c\right)=0.
$$
By Proposition \ref{cnbound}, it follows that 
$$
\lim_{N\to +\infty}\frac{t'_N}{c_{N}}=0.
$$
Therefore, by Remark \ref{remark2} we have that
$$
\lim_{ N\to +\infty}\sup_{l \in \ladder}\P\left(\frac{\tau^{N,l}}{c_{N}} > t-\frac{t'_N}{c_{N}} \right)
=e^{-t}.
$$
We conclude the proof by noting that the limits in the last equation do not depend on $u$.

\end{proof}

\begin{lemma} \label{prop7}
There exists $\alpha\in (0,1)$ and $N_{\alpha}>0$ such that for any $N > N_{\alpha}$ and any $l \in \ladder$, the following upperbound holds
$$
\P\left(\frac{\tau^{N,l}}{c_{N}} > n\right) \leq \alpha^n,
$$
 for any positive integer $n\geq 1$.
\end{lemma}

\begin{proof}
By Lemma \ref{lemmasup}, for any fixed $\alpha \in (e^{-1},1)$,  there exists $N_{\alpha}$ such that for all $N>N_{\alpha}$, 
\begin{equation} \label{p7eq1}
\sup_{u \in \st \setminus \{\nlist\}}\P\left(\frac{\tau^{N,u}}{c_{N}}>1 \right)\leq \alpha <1.
\end{equation}
 For any $l \in \ladder$ and for any $n \in \{2,3,\ldots\}$,
$$
\P\left(\frac{\tau^{N,l}}{c_{N}}>n \right)=\sum_{u \in \st \setminus\{\nlist\}} \P\left(\frac{\tau^{N,l}}{c_{N}}>n-1,U_{c_{N}(n-1)}^{N,l}=u \right)\P\left(\frac{\tau^{N,u}}{c_{N}}>1\right).
$$
Equation \eqref{p7eq1} implies that for any $N >N_{\alpha}$,
\begin{equation}\label{indu}
\P\left(\frac{\tau^{N,l}}{c_{N}}>n\right)\leq 
\alpha\P\left(\frac{\tau^{N,l}}{c_{N}}>n-1\right).
\end{equation}
We finish the proof by iterating \eqref{indu}.
\end{proof}

\begin{proof} We will now prove Theorem \ref{metastable}.

First of all, we will prove that for any sequence $(l_N \in \ladder: N\geq 2)$, the following holds
\begin{equation} \label{p1t3}
\frac{\tau^{N, l_N}}{\mathbb{E}[\tau^{N, l_N}]}\to \text{Exp}(1) \text{ in distribution, as } N \to +\infty.
\end{equation}
Considering Proposition  \ref{explim2}, we only need to show that
$$
\lim_{N \to +\infty}\frac{\mathbb{E}[\tau^{N,l_N}]}{c_{N}}=1.
$$
Actually, 
$$
\lim_{N \to +\infty}\frac{\mathbb{E}[\tau^{N,l_N}]}{c_{N}}=\lim_{ N\to +\infty}\int_0^{+\infty}\P(\tau^{N,l_N}>c_{N}s)ds.
$$
Lemma \ref{prop7} and the Dominated Convergence Theorem, allow us to put the limit inside the integral in the last term
$$
\lim_{N \to +\infty}\int_0^{+\infty}\P(\tau^{N,l_N}>c_{N}s)ds=\int_0^{+\infty}\lim_{N \to +\infty}\P(\tau^{N,l_N}>c_{N}s)ds = \int_0^{+\infty}e^{-s}ds.
$$
This and Proposition \ref{explim2} imply \eqref{p1t3}.

For any $N \geq 2$, for any $u \in S_N^{(0)}$ and for any $s>0$,
$$
\P(\tau^{N,u}> c_{N}s)=
\P(\tau^{N,u}> c_{N}s , E_{N,u})+\P(\tau^{N,u}> c_{N}s, E_{N,u}^c),
$$
where
$$
E_{N,u}=
\{\inf\{t:U_t^{N,u} \in \ladder\}\leq 1\}.
$$
For any $l \in \ladder$, by Markov property and the invariance by permutation of the process we have
\begin{equation} \label{eqt3}
\P\left(\frac{\tau^{N,l}}{c_N}> s
\right)\P(E_{N,u})\leq \P(\tau^{N,u}> c_{N}s , E_{N,u}) \leq \P\left(\frac{\tau^{N,l}}{c_N}> s-\frac{1}{c_N}\right)\P(E_{N,u}).
\end{equation}
By Corollary \ref{coro3},
$$
\lim_{N \to +\infty}\inf_{u \in S_N^{(0)}}\P(E_{N,u})=1,
$$
and then, for any sequence $(u_N \in S_N^{(0)}: N\geq 2)$,
$$
\lim_{N \to +\infty}\P(\tau^{N, u_N}> c_Ns , E_{N,u_N})=\lim_{N \to +\infty}\P(\tau^{N, u_N}> c_Ns).
$$

Proposition \ref{cnbound} and Remark \ref{remark2} implies that for any sequence $(l_N \in \ladder: N\geq 2)$,
$$
\lim_{N \to +\infty} \P\left(\frac{\tau^{N, l_N}}{c_N}> s\right)=\lim_{N \to +\infty}\P\left(\frac{\tau^{N, l_N}}{c_N}> s-\frac{1}{c_N}\right)=e^{-s}.
$$
The conclusion follows from Equation \eqref{eqt3} and by observing that the Dominated Convergence Theorem allow us to replace $c_N$ by $\mathbb{E}[\tau^{N, u_N}]$ as we did to prove that Equation \eqref{p1t3} holds.
\end{proof}

\section{Proof of Theorems \ref{teo: extinction2}, \ref{massconcentration2} and \ref{metastable2}}\label{sec:hatprocess}

In this section we show how the proofs presented before can be modified in order to prove Theorems \ref{teo: extinction2}, \ref{massconcentration2} and \ref{metastable2} concerning the  process $(\hat{U}_t^{N,u})_{t\in [0,+\infty)}$.



We will use the following notation.
Let $\hat{T}_0=0$ and for $n =1,\ldots, \hat{\mathcal{N}}^{N,u}$, let $\hat{T}_n$ denote  the successive jumping times of the process $(\hat{U}_t^{N,u})_{t\in [0,+\infty)}$, namely 
$$
\hat{T}_n= \inf\left\{t>\hat{T}_{n-1}:\hat{U}_t^{N,u} \neq \hat{U}^{N,u}_{\hat{T}_{n-1}}\right\}.
$$


We will now prove Theorem \ref{teo: extinction2}. 
\begin{proof}[Proof of Theorem \ref{teo: extinction2}]
For any $N\geq 2$ and for any $u,u' \in \st \setminus \{\nlist\}$, we have that
\begin{align*}
    &\P\Big(
    \hat{\mathcal{N}}^{N,u}\leq n+N-1+\frac{N(N-1)}{2}\ \Big\vert\ \hat{U}^{N,u}_{\hat{T}_n}=u'\Big) \geq \\ 
    &\P\left(\hat{U}_{\hat{T}_{n+N-1}}^{N,u} \in \ladder \ \Big\vert \ \hat{U}_{\hat{T}_n}^{N,u}=u'\right)\P\left(\hat{\mathcal{N}}^{N,u}\leq n+
    \frac{N^2+N-2}{2}\ \Big\vert\ \hat{U}_{\hat{T}_{n+N-1}}^{N,u} \in \ladder\right).
\end{align*}
Lemma \ref{ladderreaching} still holds for $(\hat{U}_t^{N,u})_{t\in [0,+\infty)}$. Therefore,
$$
\P\left(\hat{U}_{\hat{T}_{n+N-1}}^{N,u} \in \ladder \ \Big\vert \ \hat{U}_{\hat{T}_n}^{N,u}=u'\right) \geq [2(N-1)]^{-(N-1)}.
$$
Note that starting from a ladder list $l \in \ladder$, the process can reach the null list after $N(N-1)/2$ steps. For this, it is sufficient that the process has $N(N-1)/2$ consecutive leakage times in which 
the neuron 
with membrane potential $j$ on the list $l$
is associated to $j$ of these leakage times, for $j=1,\ldots,N-1$. 
Calling
$$
\hat{\epsilon}=\P\left(\hat{\mathcal{N}}^{N,u}\leq n+N-1+\frac{N(N-1)}{2} \ \Big\vert\ \hat{U}_{\hat{T}_{n+N-1}}^{N,u}=l\right),
$$
we conclude that $\hat{\epsilon}>0$.

Using the invariance by permutation of the process and the Markov property as in the proof of Theorem \ref{teo: extinction}, we conclude that for any $u' \in \st$,
$$
\P\Big(\hat{\mathcal{N}}^{N,u}\leq n+N-1+\frac{N(N-1)}{2}\ \Big\vert\ \hat{U}^{N,u}_{\hat{T}_n}=u'\Big) \geq [2(N-1)]^{-(N-1)}\hat{\epsilon}.
$$
The last inequality implies that for any $n\geq 1$,
$$
\P(\hat{\mathcal{N}}^{N,u} \geq n) \leq \P\left(\left(N-1+\frac{N(N-1)}{2}\right) \times \text{Geom}\big([2(N-1)]^{-(N-1)}\hat{\epsilon}\big) \geq n\right),
$$
where $\text{Geom}(r)$ denotes a random variable with geometric distribution assuming values in $\{1,2,\ldots\}$ and with mean $1/r$. 
This implies that $\P(\hat{\mathcal{N}}^{N,u}< +\infty)=1$. We prove that $\P(\hat{\tau}^{N,u}< +\infty)=1$ exactly as we did in the proof of Theorem~\ref{teo: extinction}.
\end{proof}




\begin{proposition} \label{unifconvergence2}
The following holds
$$
\lim_{N \to +\infty}\sup_{t\geq 0}\sup_{w,w' \in \W}\left\lvert\P(\hat{\tau}^{N,w}>t)-\P(\hat{\tau}^{N,w'}>t) \right\rvert=0.
$$
\end{proposition}
\begin{proof}

We can define a coupling between processes $(\hat{U}_t^{N,u})_{t\in [0,+\infty)}$ and $(\hat{U}_t^{N,v})_{t\in [0,+\infty)}$ exactly as we did in Section \ref{sec:auxiliary} for the processes $(U_t^{N,u})_{t\in [0,+\infty)}$ and $(U_t^{N,v})_{t\in [0,+\infty)}$. 
Using this coupling, we can prove Proposition \ref{unifconvergence2} 
exactly as we did with  Proposition \ref{unifconvergence}.

\end{proof}

To prove Theorem \ref{massconcentration2},
we define an auxiliary process for $(\hat{U}_t^{N,u})_{t\in [0,+\infty)}$ exactly as we did in Section \ref{sec:teo1} for the processes $(U_t^{N,u})_{t\in [0,+\infty)}$.
Using this auxiliary process, we can prove Theorem \ref{massconcentration2} 
exactly as we did with Theorem \ref{massconcentration}.

Finally, the proof of Theorem \ref{metastable2} can be done exactly as we did with Theorem~\ref{metastable}.



\section*{Acknowledgments}
This work is part of USP project {\em Mathematics, computation, language and the brain} and FAPESP project {\em Research, Innovation and Dissemination Center for Neuromathematics} (grant 2013/07699-0). The author was successively supported  by CAPES fellowship (grant \\ 88887.340896/2019-00), CNPq  fellowship (grant 164080/2021-0) and FAPESP fellowship (grant 2022/07386-0).

I thank Antonio Galves for introducing me to the topic of neuronal networks and metastability.
Finally, I thank two anonymous reviewers for their comments and suggestions.


\bibliographystyle{plainnat}
\bibliography{bibliografia}

\vspace{1cm}

\noindent
Kádmo de Souza Laxa \\ Faculdade de Filosofia, Ciências e Letras de Ribeirão Preto \\ Universidade de São Paulo \\ Av. Bandeirantes, 3900 \\
Ribeirão Preto-SP, 14040-901 \\ Brazil \\
e-mail address: \texttt{kadmolaxa@hotmail.com}

\end{document}